\documentclass[leqno]{article}

\usepackage[frenchb,english]{babel}
\usepackage[T1]{fontenc}
\usepackage{amsmath}
\usepackage{amssymb}
\usepackage{amsfonts}
\usepackage{enumerate}
\usepackage{vmargin}
\usepackage[all]{xy}
\usepackage{mathrsfs}
\usepackage{mathtools}
\usepackage{lmodern}
\usepackage{slashed}
\usepackage{ulem}
\usepackage[colorlinks=true,linkcolor=blue,pagebackref=true]{hyperref}%
\setmarginsrb{3cm}{3cm}{3.5cm}{3cm}{0cm}{0cm}{1.5cm}{3cm}
\usepackage{comment}

\newcommand{\C}{\mathbb{C}}

\newcommand{\E}{\ensuremath{\mathbb{E}}}

\newcommand{\N}{\ensuremath{\mathbb{N}}}
\newcommand{\B}{\mathrm{B}} 
\let\H\relax 
\newcommand{\H}{\mathrm{H}}
\newcommand{\K}{\ensuremath{\mathbb{K}}}

\let\L\relax 
\newcommand{\L}{\mathrm{L}}

\newcommand{\Q}{\mathbb{Q}}

\newcommand{\scr}{\mathscr}

\newcommand{\loc}{\mathrm{loc}} 

\let\cal\relax
\newcommand{\cal}{\mathcal}
\newcommand{\Z}{\ensuremath{\mathbb{Z}}}
\newcommand{\R}{\ensuremath{\mathbb{R}}}

\newcommand{\W}{\mathrm{W}}

\newcommand{\Hor}{\mathrm{H\"or}}
\newcommand{\Id}{\mathrm{Id}}

\newcommand{\la}{\langle}
\newcommand{\ra}{\rangle}

\renewcommand{\leq}{\ensuremath{\leqslant}}
\renewcommand{\geq}{\ensuremath{\geqslant}}
\newcommand{\qed}{\hfill \vrule height6pt  width6pt depth0pt}
\newcommand{\bnorm}[1]{ \big\| #1  \big\|}

\newcommand{\Bgnorm}[1]{ \Bigg\| #1  \Bigg\|}
\newcommand{\norm}[1]{\left\Vert#1\right\Vert}

\newcommand{\co}{\colon}

\newcommand{\ot}{\otimes}
\newcommand{\ovl}{\overline}

\newcommand{\dsp}{\displaystyle}
\let\i\relax 
\newcommand{\i}{\mathrm{i}}
\newcommand{\ov}{\overset}

\newcommand{\UMD}{\mathrm{UMD}}

\newcommand{\epsi}{\varepsilon}
\renewcommand{\d}{\mathop{}\mathopen{}\mathrm{d}} 
\newcommand{\e}{\mathrm{e}} 
\renewcommand{\d}{\mathop{}\mathopen{}\mathrm{d}}

\DeclareMathOperator{\dom}{dom} 
\let\Re\relax 
\DeclareMathOperator{\Re}{Re} 
\let\Im\relax 
\DeclareMathOperator{\Im}{Im} 
\DeclareMathOperator*{\esssup}{esssup} 
\DeclareMathOperator{\ord}{ord}

\DeclareMathOperator{\pv}{p.v.} 
\newtheorem{thm}{Theorem}[section]
\newtheorem{defi}[thm]{Definition}
\newtheorem{prop}[thm]{Proposition}

\newtheorem{cor}[thm]{Corollary}
\newtheorem{lemma}[thm]{Lemma}

\newtheorem{remark}[thm]{Remark}
\newtheorem{example}[thm]{Example}

\newenvironment{proof}[1][]{\noindent {\it Proof #1} : }{\hbox{~}\qed
\smallskip
}

\usepackage{tocloft}
\numberwithin{equation}{section}
\usepackage[nottoc,notlot,notlof]{tocbibind}

\let\OLDthebibliography\thebibliography
\renewcommand\thebibliography[1]{
  \OLDthebibliography{#1}
  \setlength{\parskip}{0pt}
  \setlength{\itemsep}{0pt plus 0.3ex}
}

\newcommand\reallywidehat[1]{\arraycolsep=0pt\relax%
\begin{array}{c}
\stretchto{
  \scaleto{
    \scalerel*[\widthof{\ensuremath{#1}}]{\kern-.5pt\bigwedge\kern-.5pt}
    {\rule[-\textheight/2]{1ex}{\textheight}} 
  }{\textheight} %
}{0.5ex}\\           
#1\\                 
\rule{-1ex}{0ex}
\end{array}
}

\begin{document}
\selectlanguage{english}
\title{\bfseries{Functional calculus and semilinear evolution equations for the Taibleson operator on non-Archimedean local fields}}
\date{}
\author{\bfseries{C\'edric Arhancet and Christoph Kriegler}}
\maketitle


\begin{abstract}
For any non-Archimedean local field $\mathbb{K}$ and any integer $n \geq 1$, we show that the Taibleson operator admits a bounded $\H^\infty(\Sigma_\theta)$ functional calculus on the Bochner space $\L^p(\mathbb{K}^n,Y)$ for any $\UMD$ Banach function space $Y$ and any angle $\theta > 0$, where $\Sigma_\theta=\{ z \in \mathbb{C}^*: |\arg z| < \theta \}$ and $1 < p < \infty$. Moreover, we prove that it even admits a bounded H\"ormander functional calculus of order $\frac{3}{2}$. In our study, we explore harmonic analysis on locally compact Spector-Vilenkin groups and establish the $R$-boundedness of a family of convolution operators. Our results contribute to the theory of functional calculi for operators acting on vector-valued $\L^p$-spaces over totally disconnected spaces. As an application, we obtain maximal regularity results and well-posedness for a class of evolution equations driven by the Taibleson operator.
\end{abstract}


\makeatletter
 \renewcommand{\@makefntext}[1]{#1}
 \makeatother

{
  \hypersetup{linkcolor=blue}
 \tableofcontents
}


\section{Introduction}

Over the past four decades, the realm of $q$-adic analysis has received growing attention due to its interdisciplinary applications spanning fields such as medicine, biology, and physics. For an in-depth exploration of advancements in $q$-adic analysis and its applications, we refer to the books \cite{AKS10}, \cite{Koc01}, \cite{KKZ18}, \cite{Tai75}, \cite{VVZ94}, \cite{Zun16}, \cite{Zun25} and to the surveys \cite{BrF93} and \cite{DKKVZ17}. 
 Just as phenomena in nature are traditionally modeled using complex numbers and differential equations, such as the heat equation, certain processes, including fluid dynamics in porous media \cite{KOJ16a}, \cite{KOJ16b}, macromolecules \cite{ABO04}, spin glass models \cite{ABK99}, and protein dynamics \cite{ABZ14} \cite{BiZ23} can be effectively described using non-Archimedean local fields, such as the field of $q$-adic numbers. Mathematically, the majority of these models describe the time evolution of a complex system using a so-called <<master equation>>, which is a type of parabolic equation. This equation governs the evolution of a transition function for a Markov process in an ultrametric space. In recent years, numerous topics have been explored from the perspective of ultrametric analysis, including the Navier-Stokes equations \cite{KhK20}, quantum physics \cite{VlV89}, \cite{Zun22a}, \cite{Zun24b}, quantum field theory \cite{FGZ22} and neural networks \cite{ZaZ23}, \cite{Zun24a}. We refer to \cite{Zun22b} for a concise and accessible introduction to $q$-adic analysis.

A cornerstone of ultrametric analysis, the Taibleson operator $D^\alpha$ is an ultrametric counterpart to the fractional Laplace operator $(-\Delta)^\alpha$ on $\R^n$. We refer to \cite{Sti19}, \cite{LPG20} and references therein for information on the latter operator. Recall that one method to define this linear operator is to use the formula \cite[Theorem 1]{Sti19}
\begin{equation}
\label{Frac-Lap}
(-\Delta)^\alpha f(x)
=\frac{4^\alpha\Gamma(\frac{n}{2}+\alpha)}{\pi^{\frac{n}{2}}\Gamma(-\alpha)}\pv\int_{\R^n}\frac{f(y)-f(x)}{|y-x|^{2\alpha+n}}\d y, \quad 0<\alpha <1, f \in \cal{S}(\R^n),x \in \R^n.
\end{equation}
If $\K$ is a non-Archimedean local field, for instance the field $\mathbb{Q}_q$ of $q$-adic numbers, i.e.~a non-discrete totally disconnected locally compact topological field, the Taibleson operator (or Vladimirov operator) is defined by the formula
\begin{equation}
\label{def-Vlad-Taibleson}
(D^\alpha f)(x)
\ov{\mathrm{def}}{=} \frac{1-q^\alpha}{1-q^{-\alpha - n}} \int_{\mathbb{K}^n} \frac{f(y) - f(x)}{\norm{y-x}_{\mathbb{K}^n}^{\alpha + n}} \d y, \quad \alpha > 0,\: x \in \mathbb{K}^n,
\end{equation}
for any suitable function $f \co \mathbb{K}^n \to \mathbb{C}$, where $q$ denotes the cardinality of the residue field of $\K$ and where the integral is understood in the improper sense. See \cite[Section 9.2.2]{AKS10}, \cite[(2.8)]{Koc01} for the case $n=1$ and \cite[(6) p.~331]{RoZ08} for the case $\mathbb{Q}_q^n$. Here, we use the notation
\begin{equation}
\label{norm-max}
\norm{x}_{\K^n}
\ov{\mathrm{def}}{=} \max_{1 \leq i \leq n} |x_i|_\K, \quad x=(x_1,\ldots,x_n) \in  \K^n.
\end{equation}
This operator can be seen as a Fourier multiplier with symbol $\norm{\cdot}_{\mathbb{K}^n}^\alpha$, that is,
\begin{equation}
\label{Taibleson-Fourier}
\left( D^{\alpha }f \right) (x)
=\mathcal{F}^{-1}\big( \norm{\cdot}_{\mathbb{K}^n}^\alpha\hat{f}\, \big), \quad \alpha > 0,\: x \in \mathbb{K}^n,
\end{equation}
where we use the Fourier transform $\cal{F}$. We also refer to \cite[Definition 1.12]{DeV22} for a generalization to a graded $\mathbb{K}$-Lie group. Thus, paralleling the critical role of the Laplacian in physics and geometry, the Taibleson operator stands as a central figure in the investigation of non-Archimedean analysis and in modeling natural phenomena with ultrametric spaces. 

Our main result is the following theorem, where $\Sigma_\theta \ov{\mathrm{def}}{=} \{z \in \mathbb{C}^* : |\arg(z)| < \theta\}$ (see Figure \ref{figure-sector}).

\begin{thm}
\label{Main-intro}
Let $\K$ be a non-Archimedean local field and let $n \geq 1$ be an integer. Consider a $\UMD$ Banach function space $Y$. Suppose that $1 < p < \infty$. The Taibleson operator $D^\alpha \ot \Id_Y$ admits a bounded $\H^\infty(\Sigma_\theta)$ functional calculus for any angle $\theta > 0$ on the Bochner space $\L^p(\K^n,Y)$, as well as a H\"ormander $\Hor^s_2(\R_+^*)$ functional calculus for any $s > \frac32$.
\end{thm}

In this introduction and later on in the paper, $\lesssim$ stands for an inequality up to a constant. The first part of our result says that
\begin{equation}
\label{}
\norm{f(D^\alpha \ot \Id_Y)}_{\L^p(\K^n,Y) \to \L^p(\K^n,Y)}
\lesssim \norm{f}_{\H^\infty(\Sigma_\theta)}
\end{equation} 
for any function $f$ in the algebra $\H^\infty_0(\Sigma_\theta)$ of all bounded holomorphic functions $f \co \Sigma_\theta \to \mathbb{C}$ which satisfy suitable decay estimates at 0 and at $\infty$, i.e.~for which there exist two positive real numbers $s,C>0$ such that
\begin{equation}
\label{ine-Hinfty0}
\vert f(z)\vert
\leq C\min\{|z|^s,|z|^{-s}\}, \quad z \in \Sigma_\theta.
\end{equation}
The operator $f(D^\alpha \ot \Id_Y)$ is defined by an integral over the boundary of a sector, see \eqref{2CauchySec}. The value $\omega_{\H^\infty}(D^\alpha \ot \Id_Y) = 0$ of the infimum of angles for the $\H^\infty(\Sigma_\theta)$ functional calculus obtained in Theorem \ref{Main-intro} is optimal.


In \cite[p.~264]{Tai70b} (see also \cite[Chapter VI, pp.~219 and 225]{Tai75} and \cite{Tai68,Tai70a}), Taibleson established the striking result that any essentially bounded measurable radial function \( m \colon \mathbb{K}^n \to \mathbb{C} \), i.e.~\( m(x) = \tilde{m}(\norm{x}_{\K^n}) \) for any \( x \in \mathbb{K}^n \), defines a bounded Fourier multiplier \( M_m \co \L^p(\mathbb{K}^n) \to \L^p(\mathbb{K}^n) \), \( f \mapsto \mathcal{F}^{-1}(m \hat{f}) \), on the scalar-valued \( \L^p \)-spaces \( \L^p(\mathbb{K}^n) \) for all $1 < p < \infty$. It is important to emphasize, however, that the proof of this result provides a bound of the form $\|M_m\|_{\L^p(\K^n) \to \L^p(\K^n)} \leq C_p \norm{m}_{\L^\infty(\K^n)}$ (suppose that $\norm{m}_{\L^\infty(\K^n)} \leq 1$ in the proof of \cite[Lemma 1.9, p.~223]{Tai75} and use homogeneity). As a consequence, our results in the scalar case $Y=\mathbb{C}$ can be deduced from Taibleson's theorem.  Finally, we refer to \cite[Chapter VI, pp.~219 and 225]{Tai75} for further discussion and a comparison of this result with H\"ormander's multiplier theorem in the Euclidean setting.

The result falls within the framework of symmetric contraction semigroups of operators acting on classical (or noncommutative) $\L^p$-spaces. If $(T_t)_{t \geq 0}$ is such a semigroup with generator $-A$ acting on a classical $\L^p$-space $\L^p(\Omega)$, then the operator $A$ is positive and selfadjoint on the Hilbert space $\L^2(\Omega)$. The spectral theorem thus provides a functional calculus $\mathscr{L}^\infty(\mathbb{R}_{+}) \to \B(\L^2(\Omega))$, $f \mapsto f(A)$, where the operator $f(A) \co \L^2(\Omega) \to \L^2(\Omega)$ is well-defined and bounded on the complex Hilbert space $\L^2(\Omega)$. A natural question is to determine whether $f(A)$ induces a bounded operator on the Banach space $\L^p(\Omega)$ (or more generally on the Bochner space $\L^p(\Omega,X)$) under suitable conditions on $f$. It is worth noting that the Laplacian $-\Delta$ also admits an $\H^\infty(\Sigma_\theta)$ functional calculus for any $\theta > 0$ on the Banach space $\L^p(\R^n)$ or even on any Bochner space $\L^p(\R^n,X)$ for any $\UMD$ Banach space $X$ by \cite[Theorem 10.2.25 p.~391]{HvNVW23}.

A specific instance of this problem is to determine whether the condition $f \in \Hor^s_2(\R_+^*)$ for some $s > 0$ is sufficient to ensure that the operator $f(A)$ is bounded on the Banach space $\L^p(\Omega)$. When this holds, we say that the operator $A$ admits a bounded H\"ormander functional calculus of order $s$. The function space $\Hor^s_2(\R_+^*)$ is formally defined later in \eqref{hor-2}. It is known that the Laplacian $-\Delta$ admits a H\"ormander $\Hor^s_2(\R_+^*)$ functional calculus of order $s > \frac{n}{2}$ on the space $\L^p(\R^n)$, which is a sharp condition for all $1 < p < \infty$.

This question has been extensively studied in the literature. A fundamental result was established by Christ, Mauceri and Meda, who proved in \cite[Theorem 1 p.~74]{Chr91} and \cite{MaM90} that the sub-Laplacian of a stratified Lie group $G$ admits a bounded H\"ormander functional calculus whenever $s > \frac{Q}{2}$, where $Q$ is the homogeneous dimension of $G$. An alternative proof of this result was later provided by Sikora in \cite{Sik92}. While the threshold $\frac{Q}{2}$ naturally arises in this setting, it does not always represent the most optimal bound. A significant breakthrough was made by Hebisch in \cite[Theorem 1.1 p.~233]{Heb93}, where he demonstrated a refined condition for sub-Laplacians on some products of Heisenberg-type groups. His result extends the earlier work of M\"uller and Stein \cite{MuS94}, who analyzed direct products of Heisenberg groups and Euclidean spaces. Specifically, Hebisch established that the sufficient condition $s > \frac{Q}{2}$ could be advantageously replaced by $s > \frac{d}{2}$, where $d$ denotes the topological dimension of the group. Similar improvements were obtained by Cowling and Sikora for the compact Lie group $\mathrm{SU}(2)$ in \cite[Theorem 1.1 p.~2]{CoS01}, as well as by Martini and M\"uller in \cite{MaM13} for the six-dimensional free 2-step nilpotent Lie group $N_{3,2}$. Further results along these lines can be found in \cite{MaM14b}, \cite{Mar15}, and \cite{Nie23}, which explore other classes of groups. In fact, it was later established in \cite[Theorem 3]{MaM16} that the bound $\frac{d}{2}$ is sharp for any 2-step stratified Lie group of topological dimension $d$. Moreover, Duong's approach in \cite[Theorem 3]{Duo96} highlights a connection between the $\L^1$-norm of complex-time heat kernels and the regularity order $s$ required in the H\"ormander functional calculus for sub-Laplacians on stratified Lie groups.  

On a compact Riemannian manifold $M$ of dimension $d \geq 2$ without boundary, an elliptic, self-adjoint, positive second-order differential operator is also known to admit a bounded H\"ormander functional calculus whenever $s > \frac{d}{2}$, as shown in \cite[Theorem 3.1, p.~723]{SeS89} (see also \cite[Theorem 5.3.1 p.~155]{Sog17}). Alternative proofs can be found in \cite[p.~469]{DOS02} and \cite{Xu07}. More recently, in \cite{MMN23} it is demonstrated that for a broad class of second-order differential operators associated with a sub-Riemannian structure on a smooth $d$-dimensional manifold, the condition $ s > \frac{d}{2}$ is not only sufficient but also necessary for a bounded H\"ormander functional calculus to hold.

The literature on this topic is extensive, covering a wide range of operators. Notable contributions include results on sub-Laplacians of Lie groups with polynomial growth \cite{Ale94}, operators on homogeneous metric measure spaces \cite{Blu03}, \cite{DOS02}, \cite{KuU15} and the Kohn Laplacian on some spheres \cite{CCMS17}. 
Other significant works investigate Schr\"odinger operators \cite{Heb90}, Grushin operators \cite{MaM14a}, negative generators of symmetric semigroups \cite{Med90}, the harmonic oscillator on the Moyal-Groenewold plane \cite{AHKP25}, the Hodge Laplacian on the Heisenberg group \cite{MPR07}, \cite{MPR15} and Laplacians on groups with exponential growth \cite{CGHM94}. 


It is a remarkable fact that the assumption in our result does not depend on the dimension $n$. To the best of our knowledge, this dimension-free phenomenon appears to be new.

\paragraph{Application to maximal regularity of evolution equations} 

Our study establishes the maximal regularity of the Taibleson operator $D^\alpha \ot \Id_Y$. We will outline the broad context of this concept and direct readers to \cite{Den21}, \cite{KuW04}, \cite{HvNVW23} (see also \cite{Are04}) and references therein for a comprehensive discussion of this classical topic that has received considerable attention.

If $A$ is a closed densely defined linear operator on a Banach space $X$, we can consider the following inhomogeneous autonomous Cauchy problem
\begin{equation}
\label{Cauchy-problem}
\begin{cases}
y'(t)+A(y(t))=f(t), \quad t \in ]0,T[ \\
y(0)=x_0
\end{cases},
\end{equation}
where $f \co ]0,T[ \to X$ is a function whose class belongs to the Bochner space $\L^1(]0,T[,X)$, with $0 < T < \infty$ and $x_0 \in X$. 
If $-A$ is the generator of a strongly continuous semigroup $(T_t)_{t \geq 0}$ on a Banach space $X$ then by \cite[Theorem G.3.2 p.~531]{HvNVW18} or \cite[Proposition 3.1.16 p.~118]{ABHN11} the Cauchy problem \eqref{Cauchy-problem} admits a unique \textit{mild} solution given by the <<variation of constants formula>>
\begin{equation}
\label{sol-Cauchy}
y(t) 
= T_t(x_0) +\int_{0}^{t} T_{t-s}(f(s)) \d  s, \quad 0 \leq t < T. 
\end{equation}
Suppose that $1 < q < \infty$. If the right-hand side $f$ of \eqref{Cauchy-problem} belongs to the Bochner space $\L^q(]0,T[,X)$, <<optimal regularity>> would imply that both $y'$ and $A y$ should also belong to the space $\L^q(]0,T[,X)$. We are led to say, following \cite[Definition 17.2.4 p.~577]{HvNVW23} \cite[Proposition 17.1.3 p.~572]{HvNVW23} and \cite[Proposition 17.2.11 p.~581]{HvNVW23}, that the operator $A$ has maximal $\L^q$-regularity on $]0,T[$ if for any vector-valued function $f \in \L^q(]0,T[,X)$ the mild solution of the Cauchy problem \eqref{Cauchy-problem} with $x_0=0$ is a <<strong $\L^q$-solution>>, i.e.~satisfies the following two conditions: 
\begin{enumerate}
	\item $y$ takes values in $\dom A$ almost everywhere,
	\item the function $t \mapsto A(y(t))$ belongs to the space $\L^q(]0,T[,X)$.
\end{enumerate}
A classical result of Dore \cite[Theorem 17.2.15 p.~586]{HvNVW23} says that this property implies that the operator $-A$ generates an analytic semigroup on $X$. Moreover, this property is independent of $1 < q < \infty$ by \cite[Theorem 17.2.31 p.~604]{HvNVW23}. 

A consequence of maximal $\L^q$-regularity is the is the following a priori estimate \cite[p.~577]{HvNVW23}
\begin{equation}
\label{}
\norm{Ay}_{\L^q(]0,T[,X)}
 \leq C\norm{f}_{\L^q(]0,T[ ,X)},
\end{equation}
where $C \geq 0$ denotes a constant independent of $f$. Using linearization techniques, this estimate allows an effective approach to quasilinear problems via the contraction mapping principle. For a detailed description of this technique, we refer to \cite[Chapter 18]{HvNVW23}.

It is well-known \cite[Corollary 17.3.6 p.~631]{HvNVW23} (combined with \cite[Lemma 17.2.16 p.~590]{HvNVW23}) that an operator admitting a bounded $\H^\infty(\Sigma_\theta)$ functional calculus of angle $\theta <\frac{\pi}{2}$ on a $\UMD$ Banach space $X$ has maximal $\L^p$-regularity on $]0,T[$. Consequently, our main result (Theorem \ref{Main-intro}) implies that the operator $D^\alpha \ot \Id_Y$, where $D^\alpha$ is defined in \eqref{def-Vlad-Taibleson}, has maximal $\L^q$-regularity on $]0,T[$ on the Banach space $\L^p(\K^n,Y)$ for any $1 < p,q < \infty$ and any $\UMD$ Banach function space $Y$. In the case $Y=\mathbb{C}$ and $f=0$, we obtain the equation 
\begin{equation}
\label{ultra-Heat-equation}
\frac{\partial y}{\partial t}(t,s) + D^\alpha y(t,s) 
= 0, \quad s \in \K^n, t > 0,
\end{equation}
which is the ultrametric counterpart of the classical heat equation. This equation is the <<master equation>> of some concrete models, in the case of a <<linear landscape>>, see \cite[p.~66 and p.~117]{KKZ18}, \cite[(5) p.~180]{ABKO02} and \cite[(3)]{ABO03} (see also \cite[(1) p.~viii]{Zun16}). In this framework, the ultrametric space $\K^n$ models the space of configurational states of a system, and the dynamics is given by a random walk on this space. The time evolution of the system is governed by the master equation and for each $t > 0$, the function $y \co \K^n \to \R^+$, $x \mapsto y(x,t)$ is a probability density distribution. This means that if $B$ is a measurable subset of $\K^n$ then $\int_B y(x,t) \d x$ is the probability of finding the system in $B$ at the instant $t$.

Moreover, we obtain the following consequence of Theorem \ref{Main-intro}. This result can be applied to any $\L^r$-space $Y=\L^r(\Omega)$ with $1 < r < \infty$. We refer to \eqref{angle-R-sectoriality} for the definition of the angle of $R$-sectoriality $\omega_R(B)$.

\begin{cor}
\label{cor-maximal-regularity}
Let $D^\alpha$ be the Taibleson operator acting on the Banach space $\L^p(\K^n)$ with $1<p<\infty$. Consider a $\UMD$ Banach function space $Y$ with property $(\alpha)$. If $B$ is an $R$-sectorial operator acting on $Y$ with $\omega_R(B) < \frac{\pi}{2}$, then the operator $D^\alpha \ot \Id_Y+\Id_{\L^p(\K^n)} \ot B $ has $\L^q$-maximal regularity on the Bochner space $\L^p(\K^n,Y)$ for any $1 <q <\infty$.
\end{cor}
Note that, in this context, the abstract Cauchy problem is
\begin{equation}
\label{equ-maximal-regularity}
\begin{cases}
\frac{\partial y}{\partial t} (t,s,x) +D^\alpha_s y(t,s,x) + B_x y(t,s,x) 
& = f(t,s,x) \\ 
y(0) 
& = 0
\end{cases},
\end{equation}
where the subscript indicates the variable on which the operator acts. Note that for $B = 0$ and $Y=\mathbb{C}$, this equation identifies with the equation \eqref{ultra-Heat-equation}. The additional variable $x$, taking values in a measure space $\Omega$, allows one to model internal degrees of freedom of the system, such as age structure in population dynamics, velocity variables, phenotypic traits, or chemical states. This shows that the Taibleson operator naturally fits into coupled ultrametric-classical evolution equations.
%

As an example with a relevant non-zero $B$, we can consider the case $Y = \L^r(\R^{m+1}_+)$ and a differential operator $B=\sum_{|\gamma| = 2} a_\gamma \partial_x^\gamma$ satisfying the Lopatinskij-Shapiro condition \cite[(7.4) p.~158]{KuW04} \cite[Section 4]{DeK17}, and deduce that the Cauchy problem
\begin{equation}
\label{equ-evolution-equation} 
\begin{cases}
\frac{\partial y}{\partial t} (t,s,x) + D^\alpha_s  y(t,s,x) + \sum_{|\gamma| = 2} a_\gamma \partial_x^\gamma y(t,s,x) 
& = f(t,s,x), \quad t > 0,\:s \in \K^n, \: x \in \R^{m+1}_+ \\
y(0,s,x) 
& = 0, \quad s \in \K^n,\: x \in \R^{m+1}_+ \\ 
\sum_{|\beta| = 1} b_\beta \partial_x^\beta y(t,s,x) 
& = 0, \quad t > 0,\: s \in \K^n, \: x \in \partial\R^{m+1}_+ 
\end{cases},
\end{equation}
for a given vector-valued function $f \in \L^q(]0,T[,\L^p(\K^n,\L^r(\R^{m+1}_+)))$ has $\L^q$-maximal regularity. Consequently, the parabolic evolution equation \eqref{equ-evolution-equation} admits a solution which is almost everywhere differentiable in time and takes values in the space $\dom(D^\alpha_s \ot \Id_{\L^r(\R^{m+1}_+)} + \Id_{\L^p(\K^n)} \ot B)$. 



To prove Corollary~\ref{cor-maximal-regularity}, it suffices to replace the Dunkl operator $ A^\beta$ in \cite[Corollary 4, p.~22169]{DeK17} and \cite[Proposition 4.4 p.~2171]{DeK17} with the Taibleson operator $D^\alpha$. The same argument applies, since the operator $D^\alpha \ot \Id_Y$ admits a bounded $\H^\infty(\Sigma_\theta)$ functional calculus on the Bochner space $\L^p(\K^n,Y)$ for any angle $0 < \theta < \pi$ by Theorem~\ref{Main-intro}.

\paragraph{Application to local and global existence of solutions of semilinear evolution equations} 
Consider any $\UMD$ Banach function space $Y$ and let $\alpha > 0$. Our main Theorem \ref{Main-intro} can be applied to establish local and global existence results for semilinear parabolic evolution equations that are more general than the following master equation
\begin{equation}
\label{equ-master-equation}
\frac{\partial y}{\partial t}(t,s,x) + D^\alpha y(t,s,x) = 0,
\end{equation}
where $y \in \L^q([0,T),\L^p(\K^n,Y))$.
Indeed, as a consequence of Theorem \ref{Main-intro}, the operator $D^\alpha \ot \Id_Y$, acting on the Banach space
$X = \L^p(\K^n,Y)$, admits a bounded $\H^\infty(\Sigma_\omega)$ functional calculus for any angle
$\omega \in (0,\pi)$.
In particular, it admits bounded imaginary powers for all $\omega \in (0,\pi)$ by
\cite[Theorem 15.3.20, p.~476]{HvNVW23}.
This observation entails, as corollaries, the existence of solutions for semilinear evolution equations more general than \eqref{equ-master-equation}, as shown in Theorem \ref{thm-evolution-equation}. We also refer to \cite[Chapter~18]{HvNVW23} for closely related results.

\paragraph{Approach of the proof of Theorem \ref{Main-intro}} 
Recall that the operator $-D^\alpha$ generates a Markovian semigroup $(T_t)_{t \geq 0}$ on the space $\L^\infty(\mathbb{K}^n)$. A fairly general result (Theorem \ref{Th-fun-dilation}) concerning generators of strongly continuous semigroups of positive contractions acting on $\L^p$-spaces shows that the Taibleson operator admits a bounded $\H^\infty(\Sigma_\theta)$ functional calculus on the Banach space $\L^p(\mathbb{K}^n,Y)$ for any angle $\theta > \frac{\pi}{2}$ and any $\UMD$ Banach space $Y$.

According to an abstract result (Theorem \ref{prop-KrW18}), it is sufficient to demonstrate that the semigroup is $R$-analytic of angle $\frac{\pi}{2}$ as this reduces the angle required for the functional calculus. This probabilistic condition is stronger than the property of bounded analyticity and, roughly speaking, means that the family $(T_t)_{t > 0}$ can be extended to a bounded holomorphic family $(T_z)_{\Re z > 0}$ which is <<$R$-bounded>> on each sector $\Sigma_\theta$ for any angle $0 < \theta <\frac{\pi}{2}$. We refer to \eqref{R-boundedness} for the definition of $R$-boundedness.

We then establish in Theorem \ref{Th-angle-UMD-lattice-Vilenkin} an $R$-boundedness result for a family of convolution operators on Spector-Vilenkin groups. This result is relevant since the additive group $(\K^n,+)$ of $\mathbb{K}^n$ is an abelian locally compact Vilenkin group, and since each operator $T_z$ is a convolution operator. The proof relies on a maximal inequality and is of independent interest. 

To apply this result, it is necessary to estimate the $\L^1$-norm of the kernel $K_z$ of the convolution operator $T_z$, which leads to the estimate 
\begin{equation}
\label{Final-estimates-intro}
\norm{K_z}_{\L^1(\mathbb{K}^n)}
\lesssim_{q} \frac{|z|}{\Re z}.
\end{equation}
which is in sharp contrast with the well-known analogue estimate \cite[p.~153]{Are04} (or \cite[p.~543]{HvNVW18}) saying that 
$$
\norm{h_z}_{\L^1(\mathbb{R}^n)} 
\leq \bigg(\frac{|z|}{\Re z}\bigg)^{\frac{n}{2}}, \quad \Re z > 0
$$ 
of the kernel $h_z$ of the heat semigroup on $\R^n$. This result of $R$-analyticity, enables us to obtain the boundedness of the functional calculus for any angle $\theta >0$. For the H\"ormander functional calculus, one only needs to check the conditions of the abstract result \cite[Theorem 7.1 (1) p.~424]{KrW18} (see Theorem \ref{prop-KrW18}).

\paragraph{Locally compact Vilenkin groups} 
Initially, the groups investigated by Vilenkin in \cite{Vil63} were the zero-dimensional compact \textit{abelian} groups satisfying the second axiom of countability. These groups can be described as being the topological groups such that there exists a strictly decreasing sequence $(G_n)_{n \geq 0}$ of compact open subgroups with $\cup_{n \geq 0} G_n = G$ and $\cap_{n \geq 0} G_n = \{0\}$. These groups are precisely the infinite metrizable totally disconnected compact abelian groups. Examples of such groups include the Cantor group and countable products of finite abelian groups, see Example \ref{Ex-3}.

Nowadays, the class of locally compact \textit{abelian} groups such that there exists a strictly decreasing sequence $(G_n)_{n \in \Z}$ of compact open subgroups such that $\cup_{n \in \Z} G_n = G$ and $\cap_{n \in \Z} G_n = \{0\}$ is the class of locally compact Vilenkin groups, terminology introduced by Onneweer, see \cite[Chap.~7]{SBSW15a}. Nevertheless this class of groups was already considered by some other authors before. Indeed, we note that in the classical book \cite[p.~59]{EdG77} the locally compact Vilenkin groups are called <<groups having a suitable family of compact open subgroups>> and in the book \cite{AVDR81} the locally compact Vilenkin groups are called <<locally compact 0-dimensional topologically periodic abelian groups with second axiom of countability>>. 

Our harmonic analysis results, which we establish to prove our main result, build upon previous research efforts aimed at advancing analysis on these locally compact groups. For instance, the dyadic analysis on the Cantor group (and similar groups) is well-developed, see \cite{SWS90}, \cite{SBSW15a} and \cite{SBSW15b}.

In \cite{Wei07}, the almost everywhere convergence of Banach space-valued Vilenkin-Fourier series is investigated. In \cite{KaV22}, Hardy inequalities were explored. Multiplier theory was investigated in \cite{OnQ89}. Most recently, in \cite{DHW24}, the convergence of Fejér means on some noncommutative Vilenkin groups was examined. Actually, our results are valid for a slightly broader class of \textit{not necessarily abelian} groups, which we refer to as Spector-Vilenkin groups because they were studied by Spector in \cite{Spe70}.

It is interesting to note that compact abelian Vilenkin groups frequently emerge when attempting to prove a property for all locally compact abelian groups by reducing the problem to specific types of groups through case-by-case reasoning. See, for example \cite{Arh12} and \cite[Section 7.2]{ArK23} for illustrations.


\paragraph{Structure of the paper} 
Section \ref{sec-prelim} provides background information on local fields, operator theory, functional calculus, and harmonic analysis. Section \ref{Section-Vilenkin} explores locally compact Spector-Vilenkin groups and their radial functions. In particular, we introduce the notion of least radially decreasing majorant. In Section \ref{Section-maximal}, we establish a maximal inequality related to averaging operators on locally compact Spector-Vilenkin groups. This result plays a crucial role in Section \ref{Sec-R-boundedness-of-some-family}, where Proposition \ref{Prop-pointwise-domination} sets forth a sufficient condition for the $R$-boundedness of a family of operators acting on the $\L^p$-space $\L^p(G)$ of a Spector-Vilenkin group, dominated by the maximal operator. In Theorem \ref{Th-angle-scalar-Vilenkin}, we deduce an $R$-boundedness result on some specific family of convolution operators. Section \ref{Section-Heat-p-adic-semigroup} details our proof of the estimate \eqref{Final-estimates-intro} on the kernel. We present our principal result in Theorem \ref{Main-Th-angle-0-p-adic}. Finally, in Section \ref{sec-application-existence}, we present some applications to evolution equations.

\section{Preliminaries}
\label{sec-prelim}

We begin by providing information and background on local fields.

\paragraph{Local fields}
Let $\K$ be a commutative locally compact field. For any $x \in \K^*$, we denote by $|x|$ the module of the automorphism $\K \to \K$, $y \mapsto xy$ of the additive group of $\K$. Additionally, we set $|0| \ov{\mathrm{def}}{=} 0$. According to \cite[Proposition 12 p.~VII.21]{Bou04}, the function $\K \to \R^+$, $x \mapsto |x|$ is continuous and satisfies the multiplicative property $|x y|=|x| |y|$ for any $x,y \in \K$.

Recall that a local field is a non-discrete locally compact field. It is important to note that the term local field does not always have a uniform meaning in the literature, as multiple definitions coexist. From now on, we assume that 
$\K$ is a local field. In this case, by \cite[Proposition 13 p.~VII.22]{Bou04}, the subsets $\{x \in \K : |x| \leq K \}$ for $K>0$ define a fundamental system of compact neighbourhoods of 0 in $\K$. Furthermore, by \cite[Corollary p.~VII.22]{Bou04}, a local field is a second-countable topological space. 

If for any $x,y \in \K$ we have $|x+y| \leq \max\{|x|,|y|\}$, we say, following \cite[p.~137]{RaV99}, that $\K$ is non-Archimedean and that $|\cdot|$ is an ultrametric absolute value. In this case, by \cite[Lemma 4.14 p.~145]{RaV99} the set 
\begin{equation}
\label{equ-unique-maximal-compact-subring}
\cal{O}
\ov{\mathrm{def}}{=} \{x \in \K : |x| \leq 1 \}
\end{equation}
is the unique maximal compact subring of $\K$ and is called ring of integers \cite[p.~6]{Tai75}. Furthermore, by \cite[Lemma 4.15 p.~145]{RaV99}, the set
\begin{equation}
\label{equ-unique-maximal-ideal}
\frak{p}
\ov{\mathrm{def}}{=} \{x \in \K : |x| < 1 \}
\end{equation}
is the unique maximal ideal of $\cal{O}$ and we have $\frak{p}=\beta \cal{O}$ for some element $\beta$ of $\frak{p}$. Such an element $\beta$ is called prime. The same reference states that the quotient $\cal{O}/\frak{p}$ is isomorphic to a finite field of cardinality $q=p^m$ for some prime $p$ and some positive integer $m \geq 1$. This field is called the residue field. Finally, for any $x \in \K - \{0\}$, it follows from \cite[p.~146]{RaV99} that $|x|=q^k$ for some $k \in \Z$. This implies that $|\K^*|$ is a discrete subgroup of $\R^+_*$, where $\K^* \ov{\mathrm{def}}{=} \K - \{0\}$. We will always assume that the absolute value is normalized, that is, $|\beta| = q^{-1}$. 

Now, we describe two examples of non-Archimedean local fields.

\begin{example}[fields of $q$-adic numbers] \normalfont
\label{Ex-q-adic}
Here, $q$ denotes a prime number. Recall that the field $\Q_q$ of $q$-adic numbers is defined as the completion of the field $\Q$ of rational numbers with respect to the $q$-adic norm $|\cdot|_q$ which is defined by
\begin{equation}
\label{value-q}
|x|_q
\ov{\mathrm{def}}{=} \begin{cases}
0&\text{ if }x=0\\
\frac{1}{q^k}& \text{ if }x=q^k\frac{a}{b}
\end{cases}, \quad x \in \Q,
\end{equation}
where $a$ and $b$ are integers coprime with $q$. The field $\Q_q$ has characteristic 0. We refer to the books \cite{Gou20}, \cite{Rob00} and \cite{Zun25} for more information. A particular case of \cite[Lemma 4.17 p.~146]{RaV99} (see also \cite[p.~17]{AKS10}, \cite[Proposition 2.8 p.~39]{Fol16}) says that any non-zero $q$-adic number $x$ has a unique expansion of the form
\begin{equation}
\label{dev-1}
x
=q^\gamma\sum_{j=0}^{\infty} x_j q^j,
\end{equation} 
where $\gamma \in \Z$ and $x_j \in \{0,1,\ldots,q-1\}$. Following \cite[p.~146]{RaV99}, we call $\gamma \overset{\mathrm{def}}{=} \ord_q(x)$ the order of $x$ and set $\ord_q(0) \overset{\mathrm{def}}{=} + \infty$ .
Following \cite[p.~40]{Rob00}, the fractional part $\{x\}_q$ of an element $x$ of $\Q_q$ is the element of $\mathbb{Q}$ defined by 
$$
\{x\}_q
\ov{\mathrm{def}}{=} \begin{cases}
0 & \text{if } x=0 \text{ or }\gamma \geq 0\\
\dsp q^\gamma\sum_{j=0}^{|\gamma|-1} x_jq^j &\text{if } \gamma <0
\end{cases}.
$$
The ring of integers $\cal{O}$ is the closure $\Z_q$ of the subspace $\Z$ in the space $\Q_q$ and its unique maximal ideal $\frak{p}$ is given by $\frak{p}=q\Z_q$. The residue field is $\cal{O}/\frak{p}=\Z_q/q\Z_q=\Z/q\Z$. Finally, recall that each $x \in \Z_q$ can be uniquely written as $x=\sum_{j=0}^{\infty} x_jq^j$ where $x_j \in \{0,1,\ldots,q-1\}$.
\end{example}

\begin{example}[fields of formal Laurent series over finite fields] \normalfont
\label{Ex-Laurent-series}
Let $p$ be a prime number. If $n \geq 1$ is an integer, we denote by $\mathbb{F}_{p^n}$ the finite field of cardinality $p^n$. The field of formal Laurent series over $\mathbb{F}_{p^n}$ is defined as
$$
\mathbb{F}_{p^n}((X))
\ov{\mathrm{def}}{=} \bigg\{ \sum_{j \geq m} c_j X^j : m \in \Z, c_j \in \mathbb{F}_{p^n}\bigg\}.
$$ 
This is a non-Archimedean local field of characteristic $p$. Here, the product of two Laurent series $\sum_{j \geq n} b_j X^j$ and $\sum_{j \geq m} c_j X^j$ is given by the formal Laurent series $\sum_{j} \big(\sum_k b_kc_{j-k}\big) X^j$. For any non-zero element $x=\sum_{j \geq m} c_j X^j$ of the field $\mathbb{F}_{p^n}((X))$, where $m$ is the smallest integer such that $c_m \not= 0$, the absolute value of $x$ is given by $|x| = p^{-nm}$. 

The ring of integers $\cal{O}$ of $\mathbb{F}_{p^n}((X))$ is the ring $\mathbb{F}_{p^n}[[X]] \ov{\mathrm{def}}{=} \big\{ \sum_{j \geq 0} c_j X^j : c_j \in \mathbb{F}_{p^n}\big\}$ of formal power series over the field $\mathbb{F}_{p^n}$. Its unique maximal ideal  is given by $\frak{p}=X\mathbb{F}_{p^n}[[X]]$ and the residue field $\cal{O}/\frak{p}$ is isomorphic to $\mathbb{F}_{p^n}$.
\end{example}

It is known, see e.g.~\cite[Theorem 4.12 p.~140]{RaV99}, that non-Archimedean local fields are precisely the fields $\mathbb{F}_{p^n}((X))$ of formal Laurent series over the finite fields of cardinal $p^n$ from Example \ref{Ex-Laurent-series} and the finite extensions of the fields $\Q_q$ of $q$-adic numbers from Example \ref{Ex-q-adic}, where $q$ is a prime number.

\paragraph{Pontryagin duality} Let $\K$ be a non-Archimedean local field and $n \geq 1$ be an integer. Since $(\K^n,+)$ is a locally compact abelian group, we can apply the theory of Pontryagin duality for such groups. We will describe its dual explicitly using a non-constant character $\chi \co \K \to \mathbb C$ of the additive group of $\K$. For any $x,y \in \K^n$, we let $x\cdot y \ov{\mathrm{def}}{=} x_1y_1+\cdots +x_n y_n$. Using this notation, we can define for each $x \in \K^n$ the function $\chi_x \co \K^n \to \mathbb{C}$ by  
\begin{equation}
\label{def-chi-x}
\chi_x(y) 
\overset{\mathrm{def}}{=} \chi(x \cdot y) = \chi(x_1y_1 + \cdots + x_n y_n), \quad y \in \K^n.
\end{equation}  
The map $\chi_x$ is a character of the group $\K^n$ for any $x \in \K^n$. Then by \cite[Corollaire 1 p.~II-235]{Bou19} the map $\K^n \to \widehat{\K^n}$, $x \mapsto \chi_x$ is a topological isomorphism, identifying the additive group $\K^n$ with its Pontryagin dual $\widehat{\K^n}$.


\begin{example} \normalfont
Here, we use the notations of Example \ref{Ex-q-adic}. Consider the map $\chi_{q} \co \Q_q \to \mathbb{C}$ defined by 
\begin{equation}
\label{character}
\chi_{q}(x)
\ov{\mathrm{def}}{=} \e^{2\pi \i \{x\}_{q}}, \quad x \in \Q_q.
\end{equation}
By \cite[pp.~99-100]{Fol16} or \cite[Proposition 20 p.~II.237]{Bou19}, the map $\chi_{q}$ is a non-zero character of the additive group of $\Q_q$ with kernel $\Z_q$. In particular, we have the equality $\chi_{q}(x+y)=\chi_{q}(x)\chi_{q}(y)$ for any $x,y \in \Q_q$.
\end{example}

\paragraph{Fourier transform} 
In the sequel, we will always consider the Haar measure $\mu$ on the additive group of $\K^n$ normalized by the condition $\mu(\cal{O}^n)=1$. Let again $\chi \co \K \to \mathbb C$ be a fixed non-constant character. In this context, we have a notion of Fourier transform $\cal{F}$. It is a particular case of the classical Fourier transform on a locally compact abelian group thoroughly investigated in the books \cite{Bou19} and \cite{HeR79}.  The inverse Fourier transform of a complex function $f \in \L^1(\K^n)$ is defined by
\begin{equation}
\label{inverse-Fourier-transform}
\check{f}(x) 
\ov{\mathrm{def}}{=} \int_{\K^n} \chi_x(y) f(y) \d\mu(y), \quad x \in \K^n,
\end{equation}
where $\chi_x(y)$ is defined in \eqref{def-chi-x}. In the sequel, we will write $\d y$ instead of $\d\mu(y)$ for the sake of simplicity.

Following \cite[p.~117]{Tai75}, we denote by $\scr{D}(\K^n)$ the space of linear combinations of indicator functions of spheres with complex coefficients. We equip it with its canonical topology and we denote by $\scr{D}'(\K^n)$ its topological dual, which is called the space of distributions. For any distribution $T \in \scr{D}'(\K^n)$ its Fourier transform, denoted by $\cal{F}(T)$, is the distribution defined as
$$
\la \cal{F}(T) , f \ra 
= \la T, \cal{F}(f) \ra, \quad  f \in \scr{D}(\K^n).
$$




\vspace{0.2cm}

\paragraph{Improper integrals} 
Consider a function $f \in \L^1_\loc(\K^n)$. Generalizing slightly \cite[p.~49]{AKS10} to the case of a non-Archimedean local field, we say that the improper
integral $\int_{\K^n} f(x) \d x$ exists if
\begin{equation*}
\label{}
\lim_{N \to \infty} \sum_{k=-N}^{N} \int_{\norm{x}_{\mathbb K^n} = q^{k}} f(x) \d x
\end{equation*}
exists and we set 
\begin{equation}
\label{improper}
\int_{\K^n} f(x) \d x
\ov{\mathrm{def}}{=} \lim_{N \to \infty} \sum_{k=-N}^{N} \int_{\norm{x}_{\mathbb K^n} = q^{k}} f(x) \d x
=\sum_{k=-\infty}^{\infty} \int_{\norm{x}_{\mathbb K^n} = q^{k}} f(x) \d x.
\end{equation}

We proceed with additional background on operator theory.

\paragraph{$R$-boundedness} The concept of $R$-boundedness is a probabilistic notion of boundedness for sets of operators that has gained significant importance in recent years. It generalizes the notion of bounded subset of operators acting on a Hilbert space. It has been widely applied in various areas, including maximal regularity theory for evolution equations, stochastic evolution equations and vector-valued harmonic analysis. Suppose that $1 < p < \infty$. Following \cite[Definition 8.1.1, Remark 8.1.2 p.~165]{HvNVW18}, we say that a set $\cal{F}$ of bounded operators on a Banach space $X$ is $R$-bounded provided that there exists a constant $C\geq 0$ such that for any operators $T_1,\ldots, T_n$ in $\cal{F}$ and any vectors $x_1,\ldots,x_n$ in $X$, we have
\begin{equation}
\label{R-boundedness}
\Bgnorm{\sum_{k=1}^{n} \epsi_k \ot T_k (x_k)}_{\L^p(\Omega,X)}
\leq C \Bgnorm{\sum_{k=1}^{n} \epsi_k \ot x_k}_{\L^p(\Omega,X)},
\end{equation}
where $(\epsi_{k})_{k \geq 1}$ is a sequence of independent Rademacher variables on some probability space $\Omega$. It is known that this property is independent of $p$. The best constant in \eqref{R-boundedness} is denoted by $\mathscr{R}_p(\mathcal{F})$ and is referred to as the $R$-bound of the family $\cal{F}$. By \cite[Theorem 8.1.3 p.~166]{HvNVW18}, any $R$-bounded set is necessarily bounded. It is worth noting that by \cite[Corollary 8.6.2 p.~235]{HvNVW18} $X$ is isomorphic to a Hilbert space if and only if each bounded subset of $\B(X)$ is $R$-bounded.

If $X$ is a complex Banach lattice with finite cotype, we will use a classical theorem of Maurey which asserts that we have a uniform equivalence
\begin{equation}
\label{equiv-lattice}
\Bgnorm{\sum_{k=1}^n \epsi_k \ot x_k}_{\L^p(\Omega,X)}
\approx \Bgnorm{\bigg(\sum_{k=1}^n |x_k|^{2}\bigg)^{\frac{1}{2}}}_X
\end{equation}
for any vectors $x_1,\ldots,x_n$ of $X$, see e.g.~\cite[Theorem 16.18 p.~338]{DJT95} (or \cite[Definition 8.1.1 p.~165]{HvNVW18} for the real case).

We continue by providing some background on functional calculus.

\paragraph{$\H^\infty$ functional calculus}
For a comprehensive treatment of the subject, we refer the reader to the books \cite{HvNVW18}, \cite{HvNVW23}, \cite{Haa06}, as well as the survey papers \cite{Are04}, \cite{KuW04}, \cite{LeM99}, and the references therein. We begin by recalling the main definitions. We will use the notation $\Sigma_\theta \ov{\mathrm{def}}{=} \{z \in \mathbb{C}^* : |\arg(z)| < \theta\}$  to denote the open sector of angle $\theta >0$, see Figure \ref{figure-sector}. Let $A \co \dom A \subset X \to X$ be a closed densely defined linear operator acting on a Banach space $X$. We say that $A$ is a sectorial operator of type $\omega \in (0,\pi)$ if its spectrum $\sigma(A)$ is a subset of the closed sector $\ovl{\Sigma_\omega}$ and if for any angle $\nu \in (\omega,\pi)$, the set 
\begin{equation}
\label{2Sectorial}
\Big\{zR(z,A) : z \in \mathbb{C} - \overline{\Sigma_\nu}\Big\}
\end{equation}
is bounded in the space $\B(X)$ of bounded operators acting on $X$, where $R(z,A) \ov{\mathrm{def}}{=} (z\Id-A)^{-1}$ is the resolvent operator. The operator $A$ is said to be sectorial if it is a sectorial operator of type $\omega$ for some angle $\omega \in (0,\pi)$. In this case, we can introduce the angle of sectoriality
$$
\omega(A) 
\ov{\mathrm{def}}{=} \inf\{ \omega \in (0,\pi) : A \textrm{ is sectorial of type $\omega$} \}.
$$
\begin{figure}[ht]
\label{figure-sector}
\begin{center}
\includegraphics[scale=0.4]{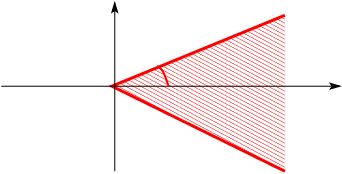}
\begin{picture}(0,0)

\put(-70,38){ $\theta$}

\put(-45,20){$\Sigma_\theta$}
\end{picture}

Figure 1: open sector $\Sigma_\theta$ of angle $\theta$
\end{center}

\end{figure}

\vspace{-0.4cm}

For any angle $\theta \in (0,\pi)$, we consider the algebra $\H^\infty_0(\Sigma_\theta)$ of all bounded holomorphic functions $f \co \Sigma_\theta \to \mathbb{C}$ for which there exist two positive real numbers $s,C > 0$ such that \eqref{ine-Hinfty0} holds. It is a (non-closed and non-dense) subalgebra of the Banach algebra 
\[ 
\H^\infty(\Sigma_\theta) 
\ov{\mathrm{def}}{=} \left\{ f \co \Sigma_\theta \to \mathbb{C} :\: f \text{ holomorphic },\: \norm{f}_{\H^\infty(\Sigma_\theta)} \ov{\mathrm{def}}{=} \sup_{z \in \Sigma_\theta} |f(z)| < \infty \right\}. 
\]
Let $A$ be a sectorial operator acting on a Banach space $X$. Consider some angle $\theta \in (\omega(A), \pi)$ and a function $f \in \H^\infty_0(\Sigma_\theta)$. Following \cite[p.~5]{LeM99}, for any angle $\nu \in (\omega(A),\theta)$ we introduce the operator
\begin{equation}
\label{2CauchySec}
f(A)
\ov{\mathrm{def}}{=} \frac{1}{2\pi \i}\int_{\partial\Sigma_\nu} f(z) R(z,A) \d z,
\end{equation}
acting on the space $X$, where the boundary $\partial\Sigma_\nu$ is parametrized by
\begin{equation}
\label{3contour}
\partial\Sigma_\nu(t)
\ov{\mathrm{def}}{=} \begin{cases}
-t \e^{\i\nu} &\text{if } t\in (-\infty,0]\\
t \e^{-\i\nu} &\text{if } t\in [0,\infty)\\
\end{cases}.
\end{equation}
In other words, the boundary of $\Sigma_\nu$ is oriented counterclockwise around the sector $\Sigma_\nu$. The sectoriality condition ensures that this integral is absolutely convergent and defines a bounded operator on $X$. Using Cauchy's theorem, it is possible to show that this definition does not depend on the choice of the chosen angle $\nu$. The resulting map $\H^\infty_0(\Sigma_\theta) \to \B(X)$, $f \mapsto f(A)$ is an algebra homomorphism.

Following \cite[Definition 2.6 p.~6]{LeM99}, we say that the operator $A$ admits a bounded $\H^\infty(\Sigma_\theta)$ functional calculus if the latter homomorphism is bounded, i.e.~if there exists a constant $C \geq 0$ such that 
\begin{equation}
\label{Def-functional-calculus}
\norm{f(A)}_{X \to X} 
\leq C\norm{f}_{\H^\infty(\Sigma_\theta)},\quad f \in \H^\infty_0(\Sigma_\theta).
\end{equation}
If the operator $A$ has dense range and admits a bounded $\H^\infty(\Sigma_\theta)$ functional calculus, then the previous homomorphism naturally extends to a bounded homomorphism $f \mapsto f(A)$ from the algebra $\H^\infty(\Sigma_\theta)$ into the algebra $\B(X)$. In this context, we can introduce the $\H^\infty$-angle
$$
\omega_{\H^\infty}(A) 
\ov{\mathrm{def}}{=} \inf\{\theta \in (\omega(A),\pi) : A \text{ admits a bounded $\H^\infty(\Sigma_\theta)$ functional calculus} \}.
$$

The following result is \cite[Theorem 10.7.12 p.~462]{HvNVW18}. It provides a broad class of examples of operators with such a functional calculus, though with an angle that may be slightly too large for applications. We refer to \cite[Chapter 4]{HvNVW16} for the definition of $\UMD$ Banach spaces.

\begin{thm}
\label{Th-fun-dilation}
Let $\Omega$ be a measure space and let $X$ be a $\UMD$ Banach space. Suppose that $1 < p < \infty$. Let $-A$ be the infinitesimal generator of a strongly continuous semigroup $(T_t)_{t \geq 0}$ of positive contractions on $\L^p(\Omega)$. Then (the closure of) the operator $A \ot \Id_X$ admits a bounded $\H^\infty(\Sigma_\theta)$ functional calculus on the Bochner space $\L^p(\Omega,X)$ for any angle $\theta >\frac{\pi}{2}$, i.e.~$\omega_{\H^\infty}(A) \leq \frac{\pi}{2}$.
\end{thm}

Recall that a closed densely defined linear operator $A$ is called $R$-sectorial of type $\omega$ if it is sectorial of type $\omega$ and if for any angle $\nu \in (\omega,\pi)$ the set appearing in \eqref{2Sectorial} is $R$-bounded. The operator $A$ is said to be $R$-sectorial if it is an $R$-sectorial operator of type $\omega$ for some angle $\omega \in (0,\pi)$. In this case, we can introduce the angle of $R$-sectoriality 
\begin{equation}
\label{angle-R-sectoriality}
\omega_R(A) 
\ov{\mathrm{def}}{=} \inf\{ \omega \in (\omega(A),\pi): A \textrm{ is $R$-sectorial of type $\omega$}\}.
\end{equation}
The next result is a particular case of \cite[Proposition 10.3.3 p.~400]{HvNVW18} and clarifies between the connection $R$-sectoriality and $R$-analyticity. Recall that a strongly continuous semigroup $(T_t)_{t \geq 0}$ of bounded operators on a Banach space $X$ is called bounded analytic (resp.~$R$-bounded analytic) if there exist an angle $\omega \in (0,\frac{\pi}{2})$ and a bounded holomorphic extension (resp.~an $R$-bounded analytic extension) $\Sigma_\omega \to \B(X)$, $z \mapsto T_z$. In the $R$-analytic case, we define the angle of $R$-analyticity by
$$
\omega_{R-\mathrm{ana}}((T_t)_{t \geq 0})
\ov{\mathrm{def}}{=} \sup \big\{\omega \in (0,\tfrac{\pi}{2}) :\text{there exists an analytic $R$-bounded extension }(T_z)_{z \in \Sigma_\omega} \big\}.
$$

\begin{thm}
\label{thm-R-anal-prime}
Let $(T_t)_{t \geq 0}$ be a strongly continuous bounded analytic semigroup on a Banach space $X$ with infinitesimal generator $-A$. If the semigroup is $R$-analytic then the operator $A$ is $R$-sectorial and
$$
\omega_R(A)
=\tfrac{\pi}{2}-\omega_{R-\mathrm{ana}}((T_t)_{t \geq 0}).
$$
\end{thm}

We will use the following result from \cite[Corollary 10.4.10 p.~422]{HvNVW18}, \cite[Theorem 5.3 p.~335]{KaW01}. Note  that every $\UMD$ Banach space has the triangular contraction property $(\Delta)$ by \cite[Corollary 7.5.9 p.~137]{HvNVW18}.  

\begin{thm}
\label{thm-R-anal}
Let $X$ be a Banach space with property $(\Delta)$. Let $A$ be a sectorial operator acting on $X$ admitting a bounded $\H^\infty(\Sigma_\theta)$ functional calculus for some angle $\theta \in (\omega(A), \pi)$. Then $A$ is $R$-sectorial and we have 
$$
\omega_R(A) 
= \omega_{\H^\infty}(A).
$$
\end{thm}

\paragraph{H\"ormander functional calculus} If $s >0$, using  the usual Sobolev norm 
\begin{equation*}
\norm{f}_{\W^{s,2}(\R)} 
\ov{\mathrm{def}}{=} \norm{(1+\xi^2)^{\frac{s}{2}}\hat{f}(\xi)}_{\L^2(\R)}
\end{equation*}
and the inhomogeneous Sobolev space $\W^{s,2}(\R) \ov{\mathrm{def}}{=} \{f \in \cal{S}'(\R) : \norm{f}_{\W^{s,2}(\R) } < \infty \}$, we introduce the norm
\begin{equation}
\label{hor-2}
\norm{f}_{\Hor^s_2(\R_+^*)}
\ov{\mathrm{def}}{=} \sup_{t > 0} \norm{x \mapsto \eta(x) \cdot f(tx) }_{\W^{s,2}(\R)} 
< \infty,
\end{equation} 
where $\eta \co ]0,\infty[ \to \R$ is any non-zero function of class $\mathrm{C}^\infty$ with compact support. The following is a direct consequence of \cite[Theorem 7.1 (1)]{KrW18}, observing that a complex Banach lattice with finite cotype has Pisier's property $(\alpha)$ by \cite[Theorem 7.5.20 p.~145]{HvNVW18}. This is a slight variant of \cite[Proposition 10.7]{AHKP25}. 

\begin{thm}
\label{prop-KrW18}
Suppose $s > 0$. Let $-A$ be the generator of a bounded analytic semigroup $(T_z)_{\Re z > 0}$ acting on a complex Banach lattice $X$ with finite cotype. Suppose that $A$ has dense range (and hence is injective), and has a bounded $\H^\infty(\Sigma_\omega)$ functional calculus for some angle $\omega \in (0,\pi)$. Assume in addition that the family
\begin{equation}
\label{family-23}
\left\{ (\cos \arg z)^\beta T_z :  \Re z > 0 \right\} 
\end{equation}
of operators is $R$-bounded over the space $X$. Then the operator $A$ admits a bounded $\Hor^s_2(\R_+^*)$ H\"ormander calculus of any order $s > \beta + \frac12$, i.e.
\begin{equation}
\label{divers-8765}
\norm{f(A)}_{X \to X} 
\lesssim \norm{f}_{\Hor^s_2(\R_+^*)}, \quad f \in \Hor^s_2(\R_+^*).
\end{equation}
\end{thm}



It is worth noting that for any angle $\theta > 0$ the space $\H^\infty(\Sigma_\theta)$ injects continuously in $\Hor^s_2(\R_+^*)$ by \cite[Proposition 9.1 b) p.~1398]{HaP23}.

\paragraph{Conditional expectations} 
Suppose that $(\Omega,\cal{A},\mu)$ is a measure space and that $(\Omega_{j})_{j \in J}$ is an at most countable and measurable partition of $\Omega$, i.e.~a partition of $\Omega$ such that $\Omega_j \in \cal{A}$ for each $j \in J$. We suppose that each subset $\Omega_j$ is of finite measure $\mu(\Omega_j) > 0$. Then a slight generalization \cite[p.~89]{Ste70} of \cite[Exercise 4 p.~227]{AbA02} says that for any $f \in \L^1(\Omega,\cal{A},\mu) \cap \L^\infty(\Omega,\cal{A},\mu)$ the conditional expectation of $f$ with respect to $\cal{A}$ is given by
\begin{equation}
\label{cond-exp}
\E(f|\cal{A})
=\sum_{j \in J} \bigg[ \frac{1}{\mu(\Omega_j)}\int_{\Omega_j} f \d \mu\bigg] 1_{\Omega_j},
\end{equation}
where $1_{\Omega_j}$ is the indicator function of the subset $\Omega_j$.

\section{Locally compact Spector-Vilenkin groups and radial functions}
\label{Section-Vilenkin}

We begin by recalling some general topological concepts. Following \cite[p.~360]{Eng89}, we say that a non-empty topological space $Y$ is hereditarily disconnected if its connected components are reduced to singletons. Unfortunately in the literature, for instance \cite[p.~11]{HeR79}, hereditarily disconnected spaces are sometimes called totally disconnected spaces. Following \cite[p.~360]{Eng89}, a topological space $Y$ is called zero-dimensional if $Y$ is a non-empty $T_1$-space and has a basis consisting of open-and-closed subsets. Finally, a space is said to be totally disconnected \cite[p.~369]{Eng89} if the quasi-component of any point $x \in Y$ consists of the point $x$ alone. By \cite[p.~369]{Eng89}, every zero-dimensional space is totally disconnected and every totally disconnected space is hereditarily disconnected. Conversely, by \cite[Theorem 6.2.9 p.~362]{Eng89}, every non-empty hereditarily disconnected locally compact space is zero-dimensional. So these notions are equivalent for the class of locally compact spaces and we will use the terminology <<totally disconnected>>.

We will use the following elementary result \cite[Corollary 3.1.5 p.~124]{Eng89}.

\begin{prop}
\label{prop-Engelking}
Let $U$ be an open subset of a compact topological space $X$. If a family $(F_k)_{k \in I}$ of
closed subsets of $X$ satisfies $\cap_{k \in I} F_k \subset U$, then there exists a finite subset $\{k_1,\ldots,k_l\}$ of $I$ such that $F_1 \cap \cdots \cap F_{k_l} \subset U$.
\end{prop}

Now, we define the class of groups on which we will focus in this paper inspired by the papers \cite[p.~3]{Vil63} and \cite[p.~61]{Spe70}. 
In the following definition, we use as order on $I$ the natural one inherited from $\R$ and we use the notation $\N \ov{\mathrm{def}}{=} \{ 0, 1, 2, 3, \ldots\}$.
\begin{defi}
\label{Def-Spector-Vilenkin}
A Hausdorff topological group $G$ is said to be a locally compact Spector-Vilenkin group if there exists a strictly decreasing sequence $(G_k)_{k \in I}$ with $I=\Z$, $I=\N$ or $I=-\N$ of compact open subgroups of $G$ such that $\cup_{k \in I} G_k = G$ and satisfying \begin{equation}
\label{}
\bigcap_{k \in I} G_k 
= \{e\}.
\end{equation}
\end{defi}

Note the following result. 

\begin{lemma} 
\label{lemma-base-vois}
Let $G$ be a Hausdorff topological group admitting a strictly decreasing sequence $(G_k)_{k \in I}$ with $I = \Z$, $I = \N$ or $I = -\N$ of compact open subgroups such that $\bigcap_{k \in I} G_k = \{ e \}$. Then the family $(G_k)_{k \in I}$ forms a basis of neighborhoods at $e$.
\end{lemma}

\begin{proof}
Each $G_k$ contains the identity element $e$ and is open, so each $G_k$ is a neighborhood of $e$. It remains to show that any neighborhood at $e$ contains one of the $G_k$. Let $U$ be such a neighborhood. Clearly, we may assume that $U$ is open. Moreover, by replacing $U$ with the smaller neighborhood $U \cap G_0$ if necessary, we can also suppose that $U \subset G_0$.

If $I = -\N$, then there is nothing to prove since $G_0 = \{ e \}$ and consequently $U = G_0$.

If $I = \N$ or $I = \Z$ we have $\bigcap_{k \in I} G_k = \{ e \} \subset U$. As $G_0$ is compact, we can use Proposition \ref{prop-Engelking} to deduce that there exists a finite subset $\{k_1,\ldots,k_l\}$ of $I$ such that $G_{k_1} \cap \cdots \cap G_{k_l} \subset U$. This means that $G_{\max\{k_1,\ldots,k_l\}} \subset U$.
\end{proof}

Let $G$ be a locally compact Spector-Vilenkin group with associated sequence $(G_k)_{k \in I}$. By Lemma \ref{lemma-base-vois}, the family $(G_k)_{k \in I}$ forms a basis of neighborhoods at $e$. Consequently, the topology on $G$ is the unique <<subgroup topology>> defined by the family $(G_k)_{k \in I}$ as defined in \cite[Example p.~223]{Bou98} or \cite{Hal50}. 

Suppose that $I=\Z$ or $I=\N$. Since $G_{k+1}$ is open in $G_k$ for each $k \in I$, the homogeneous space $G_k/G_{k+1}$ is discrete by \cite[Proposition 14 p.~231]{Bou98}. 
Furthermore $G_k$ is compact, so the homogeneous space $G_k/G_{k+1}$ is also compact by \cite[Proposition 11 p.~257]{Bou98}. 
Consequently, each $G_k/G_{k+1}$ is discrete and compact, hence finite. For any integer $k \in I$, let $d_k$ be the cardinal of the space $G_k/G_{k+1}$. Then $d_k \geq 2$. If we also have 
$$
\sup_{k \in I} |G_k/G_{k+1}|< \infty,
$$
we say that $G$ is a bounded-order locally compact Spector-Vilenkin group. If $|G_k/G_{k+1}|$ is constant, we say that $G$ is a constant-order locally compact Spector-Vilenkin group. If $I=-\N$, we leave it to the reader to adapt the previous discussion by replacing $G_{k+1}$ by $G_{k-1}$.

It is well-known that a locally compact group $G$ is totally disconnected if and only if the compact open subgroups form a basis of neighbourhoods of the identity element $e$, see e.g.~\cite[Theorem 7.7 p.~62]{HeR79} and \cite[Theorem 1.34 p.~22]{HoM13}. 
In other words, this means that $G$ has arbitrarily small compact open subgroups. If $G$ is compact, we can suppose that the subgroups are normal, again by the cited results. Consequently, a locally compact Spector-Vilenkin group is totally disconnected. Moreover, such a group is infinite since the sequence of subgroups is strictly decreasing and metrizable by the Birkhoff-Kakutani theorem \cite[p.~437]{Eng89}, which says that it is equivalent to saying that the group $G$ is first countable. 

Finally, if each subgroup $G_k$ is \textit{normal}, the locally compact Spector-Vilenkin group $G$ is pro-discrete by \cite[Proposition 6.36 p.~119]{ArK23}, and unimodular by \cite[Proposition 3 p.~20]{Bou04}, which says that the existence of a compact neighborhood of $e$ invariant by the inner automorphisms implies the unimodularity of the group.




\begin{example} \normalfont
If $I=\mathbb{N}$, we have a countable chain of strictly decreasing inclusions 
$$
\{e\} \subset \cdots \subset G_k \subset \cdots \subset G_1 \subset G_0= G.
$$ 
Since $G=G_0$, such a group is compact and we drop the word <<locally>>. By \cite[Theorem 1.34 p.~22]{HoM13} and \cite[Theorem 7.7 p.~62]{HeR79}, we can suppose that the groups $G_k$ are normal. So, we recover the notion of <<compact noncommutative Vilenkin group>> of \cite[Definition 2.1]{DeV22}. The assumption <<profinite>> is useless in \cite[Definition 2.1]{DeV22} since a Hausdorff topological group is compact and totally disconnected if and only if it is profinite by \cite[Corollary 1.2.4 p.~19]{Wil98}.
\end{example}

\begin{example} \normalfont
The subclass of locally compact \textit{abelian} Spector-Vilenkin groups with $I=\N$ is the class of groups considered in the paper \cite{Vil63}. These groups are precisely the infinite metrizable totally disconnected compact abelian groups (also stated in \cite[p.~61]{Spe70}). Indeed, Vilenkin assumes in \cite[p.~3]{Vil63} the second axiom of countability instead of our weaker assumption of metrizability, but it is equivalent for compact groups by \cite[Theorem 2.A.10 p.~15]{DH16}. Note that the dual $\hat{G}$ of such a group $G$ is discrete by \cite[Theorem 23.17 p.~362]{HeR79}. Moreover, $\hat{G}$ is a torsion group (i.e.~periodic) by \cite[Corollary 24.20 p.~383]{HeR79} and is countable by \cite[Theorem 24.15 p.~382]{HeR79}.
\end{example}

\begin{example} \normalfont
\label{Ex-3}
For any integer $m \geq 2$, we denote by $\Z/m\Z$ the cyclic group of order $m$. Let $p=(p_0,p_1,\ldots)$ be a sequence of natural numbers $p_k$ with $p_k \geq 2$. By \cite[Corollary 1.2.4 p.~19]{Wil98}, the group $G=\prod_{k=0}^{\infty} \Z/p_k\Z$ equipped with the product topology is a totally disconnected compact abelian group which is of course infinite and metrizable since the countable product of metric spaces is metrizable by \cite[Theorem 4.2.2 p.~259]{Eng89}. So it is an example of a compact \textit{abelian} Spector-Vilenkin group with $I=\N$. It is easy to define a suitable sequence $(G_k)_{k \geq 0}$.
\end{example}

\begin{example} \normalfont
\label{exa-Zq-Spector-Vilenkin}
Let $q$ be a prime number. Then the group $\Z_q=\{x \in \Q_q : |x|_q \leq 1 \}$ defined in Example \ref{Ex-q-adic} is another compact abelian Spector-Vilenkin group under the topology induced by the one of $\Q_q$. Indeed, according to \cite[Proposition 13.1.8 p.~253]{DeE14}, for each integer $k \geq 0$
$$
G_k \ov{\mathrm{def}}{=} q^k \Z_q
=\big\{x \in \Q_q : |x|_q \leq \tfrac{1}{q^k} \big\}
$$ 
is a compact open subgroup of the additive group $\Q_q$ and the sequence $(G_k)_{k \geq 0}$ is a basis of neighborhoods at 0 in the additive group $\Q_q$ and consequently also in the open subset $\Z_q$. Of course, $\bigcup_{k \geq 0} G_k = G_0 = \Z_q$ and $\bigcap_{k \geq 0} G_k=\{0\}$. Moreover, the sequence  $(G_k)_{k \geq 0}$ is clearly decreasing. 
It is even strictly decreasing. Indeed, since $|q^k|_q \ov{\eqref{value-q}}{=} \frac{1}{q^k}$ the element $q^k$ belongs to $G_k$ and does not belong to $G_{k+1}$.
\end{example}

\begin{example} \normalfont
\label{Exa-Vilenkin-group}
If $I=\Z$, we essentially recover in Definition \ref{Def-Spector-Vilenkin} the notion of <<locally compact Vilenkin group>> as introduced in \cite[Definition 1.1]{DeV22}. We have a countable chain of strictly decreasing inclusions 
$$
\{e\} \subset \cdots \subset G_{k+1} \subset G_{k} \subset \cdots \subset G.
$$ 
Indeed, the authors of \cite{DeV22} impose the additional assumptions <<unimodular>> and <<type I>>. 
\end{example}

\begin{example} \normalfont
In the classical book \cite[p.~59]{EdG77}, the locally compact \textit{abelian} Spector-Vilenkin groups with $I=\Z$ are called <<groups having a suitable family of compact open subgroups>>.
We also refer to the book \cite{AVDR81}, which appears to contain interesting information, but is unfortunately not available in English.
\end{example}

\begin{example} \normalfont
\label{Exa-Vilenkin-SBSW}
Our \textit{bounded order} locally compact \textit{abelian} Spector-Vilenkin groups with $I=\Z$ coincide with the <<locally compact Vilenkin groups>> of \cite[Chap.~7, Definition 7.1 p.~248]{SBSW15a}.
\end{example}

\begin{example} \normalfont
\label{Example-Kn}
Consider some non-Archimedean local field $\K$. Recall that the quotient $\cal{O}/\frak{p}$ is isomorphic to a finite field with cardinal $q$. Then for any integer $n \geq 1$ the additive group $(\K^n,+)$ is a locally compact Spector-Vilenkin group with $I=\Z$ in the sense of Example \ref{Exa-Vilenkin-SBSW}. Indeed, we can consider the family $(G_k)_{k \in \Z}$ of subsets defined by
\begin{equation}
\label{Gn-q-adic}
G_k
\ov{\mathrm{def}}{=} \big\{x \in \mathbb{K}^n : \norm{x}_{\mathbb{K}^n} \leq \tfrac{1}{q^k} \big\},
\end{equation}
where $\norm{\cdot}_{\mathbb{K}^n}$ is defined in \eqref{norm-max}. By essentially \cite[p.~8]{Tai75}, each $G_k$ is a compact open subgroup of $(\K^n,+)$. In particular, if $q$ is a prime number then the additive group $(\Q_q,+)$ of $q$-adic numbers is a locally compact Spector-Vilenkin group. 
\end{example} 

\begin{remark} \normalfont
In \cite[p.~1]{DHW24}, it is stated that any totally disconnected compact group is isomorphic to the complete direct product of some finite groups. This does not appear to us to be correct\footnote{\thefootnote. G.~Hong has been notified by email.} since a topological group is compact and totally disconnected if and only if it is profinite by \cite[Corollary 1.2.4 p.~19]{Wil98}.

\end{remark}


\begin{remark} \normalfont
It is stated without proof in \cite[p.~61]{Spe70}, that the locally compact \textit{abelian} Spector-Vilenkin groups coincide with the $\sigma$-compact infinite totally disconnected metrizable locally compact abelian groups  having a compact open subgroup such that the quotient is torsion.
\end{remark}

\begin{example} \normalfont
\label{exa-Vilenkin-moins-N}
If $I=-\N$, we have a countable chain of strictly decreasing inclusions 
$$
G \supset \ldots G_{-k} \supset \cdots \supset G_{-2} \supset G_{-1} \supset G_0= \{ e \}.
$$ 
The following result is written in \cite[p.~61]{Spe70} without proof. We add a proof here for the convenience of the reader. Recall that a group $G$ is said to be a torsion group if every element of $G$ has finite order.

\begin{prop}
The subclass of locally compact \textit{abelian} Spector-Vilenkin groups with $I=-\N$ is the class of torsion countable abelian discrete groups. 
\end{prop}

\begin{proof}
Suppose first that $G$ is a locally compact abelian Spector-Vilenkin group with $I = -\N$.
As $G_0= \{ e \}$ is open, the group $G$ is discrete by \cite[Proposition 5 p.~227]{Bou98}. Moreover, for any $k \in -\N$, $G_k \cong G_k / G_0$ is a subgroup of finite order. Then as $G = \bigcup_{k \in -\N} G_k$ is the countable union of finite sets, $G$ is countable. Now let $g \in G$ be arbitrary. Then there exists some $k \in -\N$ such that $g \in G_k$. As $G_k$ is finite, $g$ is of finite order, hence $G$ is torsion.

Conversely, suppose now that $G$ is a countable abelian discrete torsion group. Let $(g_k)_{k \in -\N}$ be any enumeration of the elements in $G$. Then let $G_0 \ov{\mathrm{def}}{=} \{ e \}$. For any $k \in -\N$, define inductively $G_k \ov{\mathrm{def}}{=} \langle g_1, \ldots, g_{-n} \rangle$ as the subgroup generated by the first $n$ elements of the group, where $n$ is the minimal value such that $G_{k-1} \not= G_{k}$. Note that each $G_k$ is finite (what allows us to define the sequence by recursion since $G$ is infinite). Indeed, we have $G_k = \{ g_1^{m_1} \cdots g_{-n}^{m_n} : \: 0 \leq m_j \leq (\text{order of }g_{-j})-1 \}$ since $G$ is a torsion group. Then $|G_k| \leq \prod_{j=1}^n (\text{order of }g_{-j}) < \infty$. So we obtain a strictly increasing sequence $\{e\} = G_0 \subset G_{-1} \subset G_{-2} \subset \cdots \subset G_{k} \subset\cdots \subset G$. Moreover, for any element $g \in G$, there exists $k \in -\N$ such that $g = g_k$. Hence $G = \bigcup_{k \in -\N} G_k$. Each $G_k$ is finite, hence compact and, by discreteness of $G$, open.
\end{proof}
\end{example}

\begin{example} \normalfont
Let $q \not= 2$ be a prime number. By \cite[Example 1.4]{DeV22}, the $(2n+1)$-dimensional Heisenberg group over $\mathbb{Q}_q$ is a constant-order locally compact Vilenkin group.
\end{example}

\begin{example} \normalfont
Let $q \not= 2$ be a prime number. According to \cite[Example 1.6]{DeV22}, the Engel group $\mathbb{E}_4(\mathbb{Q}_q)$ over $\mathbb{Q}_q$ is a 4-dimensional $q$-adic Lie group, which is a constant-order locally compact Vilenkin group.
\end{example}

Finally, we refer to \cite[Example 1.5]{DeV22} for the case of graded groups.


\paragraph{Crowns and radial functions}
Following \cite[Definition V.1.1 p.~61]{Spe70} (see also \cite[p.~193]{Onn88} and \cite[p.~114]{Que87}), we introduce the following definition. We leave to the reader to adapt it to the case $I=-\N$. In the next definition, the notation $A-B$ is used to denote the set-theoretic difference of subsets $A$ and $B$ of $G$.

\begin{defi}
Let $G$ be a locally compact Spector-Vilenkin group with respect to a strictly decreasing sequence $(G_k)_{k \in I}$. A crown of $G$ is a part of $G$ of the form $G_k-G_{k+1}$ for some $k \in I$. We say that a function $f \co G \to \mathbb{C}$ is radial if it is constant on each crown. 
\end{defi}

Recall that if $\phi \in \L^1(\R^n)$ the least radially decreasing majorant of $\phi$ is a function $\mathcal{R}_\phi$ defined in \cite[Theorem 2.3.8 p.~103]{HvNVW16} by $\mathcal{R}_\phi(x) \ov{\mathrm{def}}{=} \underset{|x| \leq |y|}{\mathrm{ess sup}}\, |\phi(y)|$ for almost all $x \in \R^n$. 
This function takes values in $[0,\infty) \cup \{ \infty \}$.
Moreover, it satisfies $\mathcal{R}_\phi(x) = \eta(|x|)$ for some decreasing function $\eta \co [0, \infty) \to [0,\infty) \cup \{ \infty \}$ and $|\phi(x)| \leq \mathcal{R}_\phi(x)$ for almost all $x \in \R^n$. Now, we introduce the following definition whose second part is an analogue of this notion for locally compact Spector-Vilenkin groups, greatly inspired by \cite[p.~193]{Onn88}. 

\begin{defi}
\label{Def-crown}
Let $G$ be a locally compact Spector-Vilenkin group and $(G_k)_{k \in I}$ be the associated sequence with $I=\N$ or $I=\Z$. We consider a left Haar measure $\mu_G$ on $G$ normalized by the condition $\mu_G(G_0)=1$.
\begin{enumerate}
\item Let $s \in G$ and $k \in I$. If $s \in G_k-G_{k+1}$, we let
\begin{equation}	
\label{Abs-mienne}
|s|
\ov{\mathrm{def}}{=} \mu_G(G_k).
\end{equation}
Moreover, we let $|e| \ov{\mathrm{def}}{=} 0$.
\item The radially decreasing majorant of a complex measurable function $\phi \co G \to \mathbb{C}$ is the function $\mathcal{R_\phi} \co G \to \mathbb{R}_+ \cup \{ \infty\}$ defined by
\begin{equation}
\label{radially-majorant}
\mathcal{R}_\phi(s)
\ov{\mathrm{def}}{=} \underset{|s| \leq |r|}{\esssup}\, |\phi(r)|. 
\end{equation}
\end{enumerate}
\end{defi}
We leave it to the reader to adapt these notions for $I=-\mathbb{N}$. We have almost everywhere
\begin{equation}
\label{ine-utile-789}
|\phi(s)| \leq \mathcal{R}_\phi(s).
\end{equation}
Note that $\mathcal{R}_\phi(s)$ depends only on $|s|$ and is a decreasing function of $|s|$. For any element $s \in G - \{ e \}$ and any integer $k \in I$, we have
\begin{equation}
\label{inegal-useful}
 |s| \leq \mu_G(G_k)
\iff s \in G_k.
\end{equation}

\begin{example}
\normalfont
\label{Ex-radial-bis}
Consider the additive group $G=(\K^n,+)$ of $\K^n$, where $\K$ is a non-Archimedean local field and $n \geq 1$. Equipped with the family $(G_k)_{k \in \Z}$ of subgroups defined in \eqref{Gn-q-adic}, it is a locally compact Spector-Vilenkin group by Example \ref{Example-Kn}. 
According to \cite[p.~146]{RaV99}, the function $\norm{\cdot }_{\K^n}$ takes only values in the set $q^{\Z} \cup \{0\}$. We also have 
\begin{equation}
\label{inter-98}
G_k-G_{k+1} 
\ov{\eqref{Gn-q-adic}}{=} \big\{x \in \K^n : \: \norm{x}_{\K^n} = \tfrac{1}{q^k} \big\}.
\end{equation}
According to the formula \cite[Lemma 4.1 p.~128]{Tai75} reproduced in \eqref{int-sphere}, we have 
\begin{equation}
\label{cal-inf66}
\mu_G(G_k-G_{k+1}) 
= \int_{\norm{x}_{\K^n} 
= q^{-k}} \d x
= q^{-kn}(1-q^{-n}).
\end{equation}
Using the disjoint union $G_k=(G_k- G_{k+1}) \cup (G_{k+1}- G_{k+2}) \cup \cdots$, we obtain
\begin{equation}
\label{inter-9877}
\mu_G(G_k) 
= \sum_{l=k}^\infty \mu_G(G_l-G_{l+1}) 
= (1-q^{-n})\sum_{l=k}^\infty q^{-ln} 
\ov{\eqref{cal-inf66}}{=} q^{-kn}.
\end{equation}
For any element $x \in G_k - G_{k+1}$, we infer that 
\begin{equation}
\label{compar-mod-2}
|x| 
\ov{\eqref{Abs-mienne}}{=} \mu_G(G_k) 
\ov{\eqref{inter-9877}}{=} \tfrac{1}{q^{kn}} 
\ov{\eqref{inter-98}}{=} \norm{x}_{\K^n}^n.
\end{equation}
In the particular case $n=1$ and $\K=\mathbb{Q}_q$, we obtain that $|\cdot|=|\cdot|_{q}$ and $\mu_{\mathbb{Q}_q}(G_k)=\frac{1}{q^k}$, which is also stated in \cite[(2.3) p.~40]{VVZ94}.
\end{example}

\section{Maximal inequalities}
\label{Section-maximal}


Let $I$ be a nonempty countable totally ordered set. Let $(\cal{A}_i)_{i \in I}$ be an increasing sequence of $\sigma$-algebras on a measure space $(\Omega,\cal{A},\mu)$ such that each $\cal{A}_i$ is contained in $\cal{A}$ and $\mu$ is $\sigma$-finite on each $\cal{A}_i$. If $1 < p < \infty$, the classical Doob inequality says that for any $f \in \L^p(\Omega)$ we have
$$
\norm{ \sup_{i \in I} \big| \E( f|\cal{A}_i) \big| }_{\L^p(\Omega)} 
\lesssim_p \norm{f}_{\L^p(\Omega)},
$$
where $\E( \cdot|\cal{A}_i) \co \L^p(\Omega) \to \L^p(\Omega)$ is the conditional expectation. See e.g.~\cite[Remark 1.39 p.~27]{Pis16} and \cite[Theorem 6 p.~91]{Ste70}.  Actually, if $n \geq 1$ is an integer there exists an $\ell^2_n$-extension of this inequality providing the estimate (with constant independent of $n$)
\begin{equation}
\label{Ine--Pisier-ell2}
\norm{\bigg(\sup_{i \in I} |\E( f_k |\cal{A}_i)|\bigg)_{1 \leq k \leq n}}_{\L^p(\Omega,\ell^2_n)}
\lesssim_p \bnorm{(f_k)_{1 \leq k \leq n}}_{\L^p(\Omega,\ell^2_n)},
\end{equation}
for any element $f=(f_k)_{1 \leq k \leq n}$ in the Bochner space $\L^p(\Omega,\ell^2_n)$. We refer to \cite[Remark 1.38 p.~27]{Pis16}. See also \cite[Theorem 3.2.7 p.~178]{HvNVW16}. We need a consequence of this inequality for a locally compact Spector-Vilenkin group $G$ with strictly decreasing sequence $(G_k)_{k \in I}$, where $I=\Z$, equipped with a left Haar measure. Let $(A_i)_{i \in I}$ be the family of averaging operators defined by
\begin{equation}
\label{moy-Anf}
(A_if)(s)
\ov{\mathrm{def}}{=} \frac{1}{\mu_G(G_i)} \int_{G_i} f(t^{-1} s) \d t
,\quad f \in \L^p(G),\quad s \in G,\quad i \in I.
\end{equation}
For any function $f \in \L^p(G)$, we also set
\begin{equation}
\label{Def-mmaximal-scalar}
(Mf)(s)
\ov{\mathrm{def}}{=} \sup_{i \in I}\, (A_i(|f|))(s), \quad s \in G.
\end{equation}
Now, we can prove the following result.

\begin{prop}
\label{Prop-max-ell2}
Let $G$ be a unimodular locally compact Spector-Vilenkin group equipped with a Haar measure $\mu$ with $I=\N$ or $I=\Z$. Suppose that $1 < p < \infty$. Let $f = (f_k)_{k \geq 0}$ be an element of the Bochner space $\L^p(G,\ell^2)$. We have
\begin{equation}
\label{Ine-ell2}
\bnorm{(M f_k)_{1 \leq k \leq n}}_{\L^p(G,\ell^2_n)}
\lesssim \bnorm{(f_k)_{1 \leq k \leq n}}_{\L^p(G,\ell^2_n)}.
\end{equation}
\end{prop}

\begin{proof}
Let $\cal{A}$ be the $\sigma$-algebra of Borel sets on $G$. For each integer $i \in I$, we define $\cal{A}_i$ to be the $\sigma$-subalgebra of $\cal{A}$ generated by the right cosets
$$
\{  G_is : s \in G\}
$$
of the group $G_i$. Let $i,i+1 \in I$. Since $G_i/G_{i+1}$ is finite, we can write $G_i = G_{i+1}t_1 \cup \cdots \cup G_{i+1}t_N$ for some elements $t_1,\ldots,t_N \in G_i$. Then for all $s \in G$, we have $G_is = G_{i+1}t_1 s \cup \cdots \cup G_{i+1} t_N s$. Hence $G_i s \in \cal{A}_{i+1}$. Thus, $\cal{A}_i \subset \cal{A}_{i+1}$ if the integers $i$ and $i+1$ belong to the set $I$.

It is easy to see 
that there are only countably many cosets of $G_i$ in $G$. Consider a family $(G_i s_{i,j})_{j \in J_i}$ of right coset representatives. The conditional expectation $\E(\cdot|\cal{A}_i)$ with respect to the $\sigma$-algebra $\cal{A}_i$ is given by
\begin{equation}
\label{equ-conditional-expection-proof-Prop-max-ell2}
\E(f|\cal{A}_i)
\ov{\eqref{cond-exp}}{=} \sum_{j \in J_i} \bigg[ \frac{1}{\mu(G_i s_{i,j})} \int_{ G_i s_{i,j}} f \d \mu\bigg] 1_{ G_is_{i,j}}.
\end{equation}
If $s \in G$, recall that $s \in  G_is_{i,j}$ if and only if we have $G_is_{i,j} = G_i s$. In this case, 
we obtain that
\begin{align}
\MoveEqLeft
\label{inter-567}
\E(f|\cal{A}_i)(s)         
\ov{\eqref{equ-conditional-expection-proof-Prop-max-ell2}}{=} \sum_{j \in J_i} \bigg[ \frac{1}{\mu( G_is_{i,j})} \int_{G_i s_{i,j}} f \d \mu\bigg] 1_{G_is_{i,j} }(s) 
=\frac{1}{\mu(G_i)} \int_{G_i s } f(u) \d \mu(u) \\
&\ov{t=us^{-1}}{=} \frac{1}{\mu(G_i)} \int_{G_i} f(t s) \d t
=\frac{1}{\mu(G_i)} \int_{G_i} f(t^{-1} s ) \d t
\ov{\eqref{moy-Anf}}{=} (A_if)(s) . \nonumber
\end{align}
%
Appealing to the generalization \eqref{Ine--Pisier-ell2} of Doob's inequality, for any element $(f_k)_{1 \leq k \leq n}$ of the Bochner space $\L^p(G,\ell^2_n)$, we deduce that 
\begin{align*}
\MoveEqLeft
\bnorm{(M f_k)_{1 \leq k \leq n}}_{\L^p(G,\ell^2_n)}
\ov{\eqref{Def-mmaximal-scalar}}{=} \norm{\bigg( \sup_{i \in I}\, A_i(|f_k|)\bigg)_{1 \leq k \leq n}}_{\L^p(G,\ell^2_n)} \\
&\ov{\eqref{moy-Anf}}{=} \norm{\bigg(\sup_{i \in I} |A_i(|f_k|)|\bigg)_{1 \leq k \leq n}}_{\L^p(G,\ell^2_n)} 
\ov{\eqref{inter-567}}{=} \norm{\bigg(\sup_{i \in I} |\E_i(|f_k|\, |\cal{A}_i)|\bigg)_{1 \leq k \leq n}}_{\L^p(G,\ell^2_n)}\\
&\ov{\eqref{Ine--Pisier-ell2}}{\lesssim} \bnorm{(f_k)_{1 \leq k \leq n}}_{\L^p(G,\ell^2_n)}.
\end{align*}
\end{proof}

In a similar manner to Proposition \ref{Prop-max-ell2}, we obtain the following variant where the space $\ell^2$ is replaced by any $\UMD$ Banach function space.

Given a (complete) $\sigma$-finite measure space $(\Omega',\cal{A},\mu)$, a Banach function space $Y$ over $(\Omega',\cal{A},\mu)$ \cite[Chapter I]{BeS88} (see also \cite[Definition l.b.17.~p.~28]{LiT79} for a variant) is a Banach space consisting of equivalence classes, modulo equality almost everywhere, of locally integrable functions on $\Omega'$, equipped with a norm $\norm{\cdot}_Y$, such that the following three properties hold.
\begin{enumerate}
	\item If $f \in Y$ and $g \in \L^0(\Omega')$ with $|g| \leq|f|$ a.e., then $g \in Y$ with $\norm{g}_Y \leq \norm{f}_Y$.

\item For any increasing sequence $(f_n)_{n \geq 1}$ of positive functions of $Y$ with $\sup_{n \geq 1}\norm{f_n}_Y < \infty$ such that $f_n \to f$ almost everywhere, we have $f \in Y$ and $\norm{f}_Y=\sup_{n \geq 1} \norm{f_n}_Y$.

	\item For every $A \in \cal{A}$ with $\mu(A) < \infty$ the indicator function $1_{A}$ of $A$ belongs to the space $Y$.
\end{enumerate}
Every Banach function space is a complex Banach lattice in the obvious order, i.e.~$f \geq 0$ iff $f(\omega')\geq 0$ almost everywhere.

Recall \cite[Theorem 3 p.~212]{Rub86} \cite{Bou84} that a complex Banach lattice $Y$ is $\UMD$ if and only if $Y$ and $Y^*$ have the Hardy-Littlewood property. Let $Y = Y(\Omega')$ be a Banach function space. Further, let $(\cal{A}_i)_{i \in I}$ be a filtration on a measure space $(\Omega,\cal{A},\mu)$ such that $\mu$ is $\sigma$-finite on each $\cal{A}_i$ and denote by $\E_i$ the corresponding conditional expectation operators. 

Suppose that $Y$ has the Hardy-Littlewood property, meaning that the dyadic lattice maximal function $M_d(f)(x,\omega') = \sup_{I \ni x, \: I \in \mathcal{D}} \frac{1}{|I|} \left| \int_T f(y,\omega') \d y \right|$ is bounded on the Bochner space $\L^p([0,1),Y)$, where $\mathcal{D}$ is the family of dyadic intervals on the unit interval $[0,1)$ for all $1 < p < \infty$. Then according to \cite[Lemma 3.4 p.~19]{DKK21} the lattice maximal function
\begin{equation}
M(f)(x,\omega') 
\ov{\mathrm{def}}{=} \sup_{i \in I} |\E_i f(x,\omega')| ,
\quad x \in \Omega, \quad \omega' \in \Omega' ,
\end{equation}
defines a bounded operator on the Bochner space $\L^p(\Omega,Y)$ for any $p \in (1,\infty)$. For a locally compact Spector-Vilenkin group $G = \bigcup_{i \in I} G_i$, we let
\begin{equation}
\label{equ-def-mmaximal-UMD}
M(f)(s,\omega') 
\ov{\mathrm{def}}{=} \sup_{i \in I}\, (A_i(|f(\cdot,\omega')|))(s), \quad f \in \L^p(G,Y(\Omega')),\: s \in G, \omega' \in \Omega'.
\end{equation}

\begin{prop}
\label{Prop-max-UMD-lattice}
Let $G$ be a unimodular locally compact Spector-Vilenkin group with $I=\Z$ or $I=\N$. Suppose that $1 < p < \infty$. Let $Y = Y(\Omega')$ be a $\UMD$ Banach function space. For any element $f$ of the Bochner space $\L^p(G,Y)$, we have
\[ 
\norm{ M(f) }_{\L^p(G,Y)} 
\lesssim \norm{f}_{\L^p(G,Y)}. 
\]
\end{prop}

\begin{proof}
We consider the $\sigma$-algebras $\cal{A}$ and $\cal{A}_i$ as in the proof of Proposition \ref{Prop-max-ell2}. Note that $Y$ has the Hardy-Littlewood property. Using the previously described result \cite[Lemma 3.4 p.~19]{DKK21}, we obtain
\[ 
\norm{ \sup_{i}\, (A_i(|f(\cdot,\omega')|)(s)}_{\L^p(G,Y)} 
\lesssim \norm{f}_{\L^p(G,Y)}. 
\]
\end{proof}

We leave it to the reader to adapt these results for $I=-\N$. 

\section{$R$-boundedness of some family of convolution operators}
\label{Sec-R-boundedness-of-some-family}

If $G$ is a locally compact group equipped with a left Haar measure $\mu_G$,  recall that the convolution product of two functions $f$ and $g$ is given, when say $f,g \in \L^1(G)$ by
\begin{equation}
\label{Convolution-formulas}
(g*f)(s)
\ov{\mathrm{def}}{=} \int_G g(t)f(t^{-1}s) \d \mu_G(t), \quad s \in G.
\end{equation}
Moreover, the convolution exists as an integral for $f,g$ belonging to some different function spaces, see the results below.
The following is an analogue of \cite[Proposition 2.3.9 p.~104]{HvNVW18} and \cite[p.~82 and p.~84]{Gra14a}. See also \cite{Lor19} for related results. Recall that the radially decreasing majorant $\mathcal{R}_\phi$ is defined in \eqref{radially-majorant}. 

\begin{prop}
\label{Prop-majoration-decreasing}
Let $G$ be a locally compact Vilenkin group. Let $f \in \L^1_\loc(G)$ and $\phi \in \L^1(G)$. For almost all $s \in G$, we have
\begin{equation}
\label{majo-useful-66}
|(\phi*f)(s)|
\leq \norm{\mathcal{R}_\phi}_{\L^1(G)}(M f)(s)
\end{equation}
whenever the right-hand side is finite.
\end{prop}

\begin{proof}
Since $\mathcal{R}_\phi(s)$ depends only on $|s|$ and is a decreasing function of $|s|$, there exists some sequence $(u_k)_{k \in \Z}$ of positive real numbers such that
\begin{equation}
\label{Phi-as-sum}
\mathcal{R}_\phi(s)
=\sum_{k \in \Z,|s| \leq \mu_G(G_k)} u_k, \quad s \in G.
\end{equation}
Using Fubini's theorem in the second equality, we see that
\begin{align}
\MoveEqLeft
\label{Equa-inter-123}
\sum_{k \in \Z} \mu_G(G_k) u_k 
=\sum_{k \in \Z} \bigg(\int_G 1_{G_k} \d\mu_G\bigg) u_k  \\
&=\int_G \sum_{k \in \Z} 1_{G_k}(s)u_k \d s \nonumber
\ov{\eqref{inegal-useful}}{=} \int_G \sum_{|s| \leq \mu_G(G_k)} u_k \d s \nonumber 
\ov{\eqref{Phi-as-sum}}{=} \int_G \mathcal{R}_\phi(s) \d s \nonumber \\
&=\norm{\mathcal{R}_\phi}_{\L^1(G)}. \nonumber   
\end{align}
For almost all $s \in G$, using again Fubini's theorem in the fourth equality, we obtain
\begin{align*}
\MoveEqLeft
|(\phi*f)(s)|           
\ov{\eqref{Convolution-formulas}}{=} \left|\int_{G} \phi(t)f(t^{-1}s) \d t\right| 
\leq \int_{G} |\phi(t)| |f(t^{-1}s)| \d t \\
&\ov{\eqref{ine-utile-789}}{\leq} \int_{G} \mathcal{R}_\phi(t) |f(t^{-1}s)| \d t
\ov{\eqref{Phi-as-sum}}{=} \int_{G} \bigg(\sum_{|t| \leq \mu_G(G_k)} u_k\bigg)|f(t^{-1}s)| \d t \\
&\ov{\eqref{inegal-useful}}{=} \int_{G} \sum_{k \in \Z} 1_{G_k}(t) u_k |f(t^{-1}s)| \d t
= \sum_{k \in \Z} \bigg(\int_{G_k} |f(t^{-1}s)| \d t\bigg)u_k \\
&=\sum_{k \in \Z} \mu_G(G_k) \bigg(\frac{1}{\mu_G(G_k)} \int_{G_k} |f(t^{-1}s)| \d t\bigg)u_k \\
&\ov{\eqref{Equa-inter-123}}{\leq} \norm{\cal{R}_\phi}_{\L^1(G)} \sup_{k \in \Z} \frac{1}{\mu_G(G_k)} \int_{G_k} |f(t^{-1}s)| \d t \\
&\ov{\eqref{moy-Anf}}{=} \norm{\cal{R}_\phi}_{\L^1(G)} \sup_{k \in \Z}\, (A_k(|f|))(s)
\ov{\eqref{Def-mmaximal-scalar}}{=} \norm{\cal{R}_\phi}_{\L^1(G)}(M f)(s).
\end{align*} 
\end{proof}

The following is an analogue of \cite[Proposition 8.2.1 p.~184]{HvNVW18} for locally compact Spector-Vilenkin groups. Recall that $M$ is defined in \eqref{Def-mmaximal-scalar}.

\begin{prop}
\label{Prop-pointwise-domination}
Let $G$ be a unimodular locally compact Spector-Vilenkin group. Suppose that $1 < p < \infty$. The family $\mathscr{T} \ov{\mathrm{def}}{=} \{T \co \L^p(G) \to \L^p(G) : |T(f)| \leq M(|f|)\}$ is $R$-bounded with
\begin{equation}
\label{}
\mathscr{R}_p(\mathscr{T})
\leq C_{G,p}
\end{equation}
for some constant $C_{G,p}$. Moreover, if $Y$ is a $\UMD$ Banach function space, then the family $\mathscr{T}_Y \ov{\mathrm{def}}{=} \{ T \co \L^p(G,Y) \to \L^p(G,Y) : |T(f)(s,\omega)| \leq M|f|(s,\omega) \}$ is $R$-bounded over the Bochner space $\L^p(G,Y)$.
\end{prop}

\begin{proof}
Let $n \geq 1$ be an integer and consider $T_1,\ldots,T_n \in \mathscr{T}$. Using Proposition \ref{Prop-max-ell2} in the fourth step, we obtain for any functions $f_1,\ldots,f_n \in \L^p(G)$
\begin{align*}
\MoveEqLeft
\norm{\sum_{k=1}^{n} \epsi_k \ot T_k(f_k) }_{\L^p(\Omega,\L^p(G))}     
\ov{\eqref{equiv-lattice}}{\approx}  \norm{\bigg(\sum_{k=1}^{n} |T_k(f_k)|^2 \bigg)^{\frac{1}{2}}}_{\L^p(G)} \\
&\leq \norm{\bigg(\sum_{k=1}^{n} (M(f_k))^2 \bigg)^{\frac{1}{2}}}_{\L^p(G)} \\
&=\norm{\big(M(f_k)\big)_{1 \leq k \leq n}}_{\L^p(G,\ell^2_n)} 
\ov{\eqref{Ine-ell2}}{\lesssim_{G,p}} \norm{(f_k)_{1 \leq k \leq n}}_{\L^p(G,\ell^2_n)} \\
&=\norm{\bigg(\sum_{k=1}^{n} |f_k|^2 \bigg)^{\frac{1}{2}}}_{\L^p(G)}
\ov{\eqref{equiv-lattice}}{\approx} \norm{\sum_{k=1}^{n} \epsi_k \ot f_k}_{\L^p(\Omega,\L^p(G))}.
\end{align*}
We turn to the last sentence of the statement. Since a $\UMD$ Banach space has finite cotype by \cite[Proposition 7.3.15 p.~99]{HvNVW18}, the equivalence \eqref{equiv-lattice} also holds in the space $\L^p(G,Y)$. Note that when $Y = Y(\Omega)$ is a $\UMD$ Banach function space, then the vector-valued space 
$$
Y(\Omega,\ell^2) 
\ov{\mathrm{def}}{=} \{ f \co \Omega \to \ell^2 :\: f \text{ is strongly measurable and } \omega \mapsto \norm{f(\omega)}_{\ell^2} \text{ belongs to } Y\},
$$
with norm $\norm{f}_{Y(\Omega,\ell^2)} \ov{\mathrm{def}}{=} \bnorm{ \norm{F(\cdot)}_{\ell^2} }_Y$, is also a $\UMD$ Banach function space according to \cite[Lemma 2.9 p.~4654-4655]{DeK23}. Then argue as before, this time using Proposition \ref{Prop-max-UMD-lattice} instead of Proposition \ref{Prop-max-ell2}, with the $\UMD$ Banach function space $Y(\Omega,\ell^2)$ .
\end{proof}

 If $1 < p < \infty$, note that a sufficient condition for the $R$-boundedness of a family $(C_K)_{K \in \cal{K} }$ of integral operators $C_K \co \L^p(\R^n,X) \to \L^p(\R^n,X)$, $f \mapsto K*f$ on the Bochner space $\L^p(\R^n, X)$ is well-known if $X$ is $\UMD$ \textit{Banach function space}. Indeed, if $M$ is the Hardy-Littlewood maximal operator and if
\begin{equation*}
\cal{K} 
\ov{\mathrm{def}}{=} \big\{ K \in \L^1(\R^n) : |K| * |f| \leq Mf \text{ a.e. for all simple function } f  \co \R^n \to \mathbb{C} \big\}
\end{equation*}
then the family $(C_K)_{K \in \cal{K} }$ of operators is $R$-bounded on the Bochner space $\L^p(\R^n,X)$, see e.g.~\cite[Proposition 3.1 p.~371]{Lor19}
. As observed in \cite[p.~371]{Lor19}, this class of kernels contains any radially decreasing $K \in \L^1(\R^n)$ with $\norm{K}_{\L^1(\R^n)} \leq 1$. The following is an analogue of \cite[Proposition 8.2.3 p.~184]{HvNVW18}.

\begin{thm}
\label{Th-angle-scalar-Vilenkin}
Let $G$ be a unimodular locally compact Spector-Vilenkin group. Let $\scr{K}$ be the subset of all $K \in \L^1(G)$ such that
\begin{equation}
\label{inter11}
\int_G \underset{|s| \leq |r|}{\esssup}\, |K(r)| \d s
\lesssim 1.
\end{equation}
Suppose that $1 < p < \infty$. Then the family $\scr{T}_{\scr{K}} \ov{\mathrm{def}}{=} \{f \mapsto K*f : \: K \in \scr{K}\}$ of convolution operators acting on the space $\L^p(G)$ is $R$-bounded and
\begin{equation}
\label{}
\scr{R}_p(\scr{T}_{\scr{K}})
\leq C_{G,p}
\end{equation}
for some constant $C_{G,p}$. 
\end{thm}

\begin{proof}
Let $f \in \L^p(G)$ and consider some $K \in \scr{K}$. Recall that $\mathcal{R}_K(s) \ov{\eqref{radially-majorant}}{=} \underset{|s| \leq |r|}{\esssup}\, |K(r)|$. By Proposition \ref{Prop-majoration-decreasing}, we have for almost all $s \in G$
$$
|(K*f)(s)| 
\ov{\eqref{majo-useful-66}}{\leq} \norm{\mathcal{R}_K}_{\L^1(G)} (M f)(s)
= \bigg(\int_G \underset{|s| \leq |r|}{\esssup}\, |K(r)| \d s\bigg) (M f)(s)
\ov{\eqref{inter11}}{\lesssim} (M f)(s).
$$
Now, it suffices to use Proposition \ref{Prop-pointwise-domination}.
\end{proof}

In a similar manner to the proof of Theorem \ref{Th-angle-scalar-Vilenkin}, by applying Proposition \ref{Prop-max-UMD-lattice} instead of Proposition \ref{Prop-max-ell2}, we obtain the following <<$\UMD$ Banach function space-valued>> variant.

\begin{thm}
\label{Th-angle-UMD-lattice-Vilenkin}
Let $G$ be a unimodular locally compact Spector-Vilenkin group. Let $Y$ be a $\UMD$ Banach function space. Suppose that $1 < p < \infty$. Let $\scr{K}$ be the same subset of $\L^1(G)$ as in the statement of Theorem \ref{Th-angle-scalar-Vilenkin} and $\scr{T_{\scr{K}}}$ the associated family of convolution operators acting on the Bochner space $\L^p(G,Y)$. Then this family is $R$-bounded and
\begin{equation}
\label{}
\scr{R}_p(\scr{T}_{\scr{K}})
\leq C_{G,p,Y}
\end{equation}
for some constant $C_{G,p,Y}$.
\end{thm}


\section{Application to functional calculus}
\label{Section-Heat-p-adic-semigroup}

In this section, we consider a non-Archimedean local field $\mathbb K$ and an integer $n \geq 1$. By Example \ref{Example-Kn}, the additive group of $\mathbb K^n$ is a locally compact Spector-Vilenkin group. Recall the notation $\norm{x}_{\mathbb K^n} = \max_{1 \leq i \leq n} |x_i|_{\mathbb K}$ from \eqref{norm-max} for any $x=(x_1,\ldots,x_n)$ in $\K^n$.
We further fix a non-constant character $\chi \co \mathbb K \to \C$ of $\mathbb K$.  
Recall that $q$ is the cardinal of the quotient $\cal{O}/\frak{p}$. See \cite[p.~15]{Tai75}. We will use the notation $x \cdot y \ov{\mathrm{def}}{=} x_1y_1 + \cdots + x_n y_n$, previously introduced.

\paragraph{Convergent integrals}
A certain number of integrals over $\mathbb K^n$ considered in this section do not converge absolutely. However, this can be repaired by restricting the domain of integration first to a sphere $\{ y \in \mathbb K^n : \norm{y}_{\mathbb K^n} = q^{-k}  \}$ where $k \in \Z$, and then summing over $k \in \Z$. Here, the series over $k \in \Z$ will always be absolutely convergent. Let us turn to the details.

For any integer $k \in \Z$ and any element $x \in \K^n$ we need the value of the classical integral\footnote{\thefootnote. In the second case, the integral is equal to $-\norm{x}_{\mathbb K^n}^{-n}$.} \cite[Lemma 4.1 p.~128]{Tai75}
%
\begin{equation}
\label{equ-proof-prop-norm-alpha-negative-definite}
\int_{\norm{y}_{\K^n} = q^{-k}} \chi(x\cdot y) \d y 
= \begin{cases}
q^{-kn}(1-q^{-n}) & \text{ if }  \norm{x}_{\K^n} \leq q^{k} \\ 
-q^{-n(k+1)}  & \text{ if } \norm{x}_{\K^n} = q^{k+1} \\ 0 & 
\text{ if } \norm{x}_{\K^n} \geq q^{k+2}
\end{cases},
\end{equation}
where we refer to the notion of improper integrals from \eqref{improper}.
In particular with $x=0$, we obtain the formula
\begin{equation}
\label{int-sphere}
\int_{\norm{y}_{\K^n} = q^{-k}}  \d y 
= q^{-kn}(1-q^{-n})
=q^{-kn}-q^{-(k+1)n},
\end{equation}
see also \cite[(15.4) p.~64]{Vla99}. We also  need the following lemma. 

\begin{lemma}
\label{lem-integrals-over-spheres}
If $x \in \K^n$ and $\alpha >0$, then the measurable function $F \co \mathbb{K}^n \to \C$, $y \mapsto \frac{1 - \chi(x \cdot y)}{\norm{y}_{\mathbb K^n}^{\alpha + n}}$ admits an integral over $\mathbb K^n - \{ 0 \}$ which converges in the improper sense of \eqref{improper} with an absolutely convergent series. 
\end{lemma}

\begin{proof}
According to \eqref{equ-proof-prop-norm-alpha-negative-definite} and to \eqref{int-sphere} we have
\begin{align*}
\MoveEqLeft
\sum_{k \in \Z} \int_{\norm{y}_{\K^n} = q^{-k} } F(y) \d y 
=\sum_{k \in \Z} \int_{\norm{y}_{\K^n} = q^{-k} } \frac{1 - \chi(x \cdot y)}{\norm{y}_{\mathbb K^n}^{\alpha + n}} \d y \\
&=\sum_{k \in \Z} \bigg(\int_{\norm{y}_{\K^n} = q^{-k} } \frac{1}{\norm{y}_{\mathbb K^n}^{\alpha + n}}\d y- \int_{\norm{y}_{\K^n} = q^{-k} }\frac{\chi(x \cdot y)}{\norm{y}_{\mathbb K^n}^{\alpha + n}} \d y \bigg) \\
&=\sum_{k \in \Z} \bigg(q^{k(\alpha + n)}\int_{\norm{y}_{\K^n} = q^{-k} } \d y - q^{k(\alpha + n)}\int_{\norm{y}_{\K^n} = q^{-k} } \chi(x \cdot y) \d y \bigg)\\
&\ov{\eqref{int-sphere}\eqref{equ-proof-prop-norm-alpha-negative-definite}}{=} \sum_{k \in \Z} q^{k(\alpha + n)}\big(q^{-kn}(1-q^{-n})\big) - \sum_{\norm{x}_{\K^n} \leq q^{k}} q^{k(\alpha + n)}\big(q^{-kn}(1-q^{-n})\big) \\
&+ \norm{x}_{\K^n}^{-n}\norm{x}_{\K^n}^{\alpha + n} q^{-(\alpha+n)} \\
&=\sum_{\norm{x}_{\mathbb K^n} \geq q^{k+1}} q^{k\alpha}(1-q^{-n})+\norm{x}_{\K^n}^{\alpha} q^{-\alpha-n}.
\end{align*}
Again, this series is absolutely convergent since we have $\norm{x}_{\K^n}=q^{n_0}$ for some integer $n_0 \in \Z$.
\end{proof}

\paragraph{Gelfand-Graev Gamma function}
Let $n \geq 1$ be an integer. We will use an $n$-dimensional generalization $\Gamma_q^{(n)}$ of the Gelfand-Graev Gamma function $\Gamma_q$.  This function is introduced in \cite[p.~128]{Tai75} and  defined by
\begin{equation}
\label{Gamma-formula-n}
\Gamma_q^{(n)}(z)
\ov{\mathrm{def}}{=} \frac{1-q^{z-n}}{1-q^{-z}}, \quad z \in \mathbb{C}-\tfrac{2 \pi \i}{\ln(q)} \mathbb{Z}.
\end{equation}
The function $\Gamma_q^{(n)}$ is meromorphic. Its set of zeros\footnote{\thefootnote. It is stated in \cite[p.~128]{Tai75} that the function $\Gamma_q^{(n)}$ has a unique zero and a unique pole,  which is false.}  (all simple) consists of $n + \frac{2\pi \i}{\ln(q)}\mathbb{Z}$ and its set of poles (all simple) is given by $z \in \frac{2 \pi \i}{\ln(q)} \mathbb{Z}$. If $z$ is a complex number such that $\Re z >0$ and if $u \in \mathbb{K}^n$ satisfies $|u|=1$, it is shown in \cite[Theorem 4.2 p.~128]{Tai75} that
\begin{equation}
\label{Gamma-function-n}
\Gamma_q^{(n)}(z)
=\int_{\mathbb{K}^n} \norm{x}_{\mathbb K^n}^{z-n} \ovl{\chi}(u \cdot x) \d x,
\end{equation}
where the integral is understood in the improper sense. Recall that by \cite[Remark p.~128]{Tai75} we have
\begin{equation}
\label{prod-Gamma}
\Gamma_q^{(n)}(z)\Gamma_q^{(n)}(n-z)
=1.
\end{equation}
If $n=1$, we recover the Gelfand-Graev Gamma function $\Gamma_q$ which admits the simple formula
\begin{equation*}
\label{Gamma-formula}
\Gamma_q(z)
=\frac{1-q^{z-1}}{1-q^{-z}}, \quad z \in \mathbb{C}-\tfrac{2 \pi \i}{\ln(q)} \mathbb{Z}.
\end{equation*}
If $z$ is a complex number such that $\Re z >0$, formula \eqref{Gamma-function-n} gives in particular
\begin{equation*}
\label{Gamma-function}
\Gamma_q(z)
\ov{\mathrm{def}}{=} \int_{\mathbb{K}} |x|^{z-1} \ovl{\chi}(x) \d x, 
\end{equation*}
which is stated in \cite[(a) p.~49]{Tai75}. We also refer to \cite[p.~16]{BrF93} for more information on this function.

\paragraph{L\'evy-Khinchin formula} Recall that a function $\psi \co G \to \mathbb{C}$ on a locally compact abelian group $G$ with neutral element $e$ is said to be conditionally negative definite \cite[Proposition 7.5 p.~40]{BeF75}  \cite[Problem 2 p.~12]{Dieu78} (see also \cite[Definition C.4.17 p.~357]{BHV08}) if $\psi(e) \geq 0$, $\ovl{\psi(s)}=\psi(s^{-1})$ for any $s \in G$ and if any integer $n \geq 1$, any elements $\gamma_1,\ldots,\gamma_n \in G$ and any complex numbers $\xi_1,\ldots,\xi_n \in \mathbb{C}$ the condition $\sum_{j=1}^n \xi_j=0$ implies that
$$
\sum_{i,j=1} \psi(s_j^{-1} s_k) \ovl{\xi_j}\xi_k
\geq 0.
$$
This notion plays a central role in harmonic analysis and probability theory, since conditionally negative definite functions are precisely those that generate convolution semigroups of positive bounded measures, as proved in \cite[Theorem 8.3, p.~49]{BeF75}.

A character $\chi$ of a locally compact abelian group $G$ is a positive definite function by \cite[3.5 p.~13]{BeF75}. Note in particular that this implies by \cite[Corollary 7.7 p.~41]{BeF75} and \cite[Proposition 7.4 (ii) p.~40]{BeF75} that the function $G \to \R$, $s \mapsto 1-\Re \chi(s)$ is conditionally negative definite. These elementary examples will serve as building blocks in what follows.

Let $n \geq 1$ and $0 < \alpha <2$. Recall the well-known classical formula \cite[p.~184]{BeF75}:
\begin{equation}
\label{}
\norm{x}_{\R^n}^\alpha
=\frac{\alpha2^{\alpha-1}\Gamma(\frac{\alpha+n}{2})}{\pi^{\frac{n}{2}}\Gamma(1-\frac{\alpha}{2})} \int_{\R^n-\{0\}} \frac{1-\cos \la x,y \ra}{\norm{y}_{\R^n}^{\alpha+n}} \d y, \quad x \in \R^n,
\end{equation}
where $\norm{x}_{\R^n} \ov{\mathrm{def}}{=} (\sum_{i=1}^{n} x_i^2)^{\frac{1}{2}}$. This is a L\'evy-Khinchin representation of the conditionally negative definite function $\norm{\cdot}_{\R^n}^\alpha$ as a <<continuous sum>> of elementary conditionally negative definite functions, which is a particular case of \cite[Corollary 18.20 p.~184]{BeF75}. Replacing $\R$ by a non-Archimedean local field $\K$, we will establish in Proposition \ref{prop-norm-alpha-negative-definite} the folklore analogue of this formula for $0 < \alpha <\infty$ instead of $0 < \alpha < 2$ with sharp contrast, for the sake of completeness.


According to \cite[p.~129]{Tai75}, if $\Re \alpha >0$ the function $\norm{\cdot}_{\K^n}^{\alpha-n}$ determines a distribution on $\K^n$ acting as
$$
\big\la \norm{\cdot}_{\K^n}^{\alpha-n}, f \big\ra
=\int_{\K^n} \norm{x}_{\K^n}^{\alpha-n}f(x) \d x, \quad f \in \scr{D}(\K^n).
$$ 
Actually, the map $\alpha \mapsto \norm{\cdot}_{\K^n}^{\alpha-n}$ extends to a meromorphic distribution-valued function on the subset $\mathbb{C}-(n+\frac{2\pi \i}{\ln q}\Z)$ given by
\begin{equation*}
\big\la \norm{\cdot}_{\K^n}^{\alpha-n} ,f \big\ra 
\ov{\mathrm{def}}{=} \frac{1-q^{-n}}{1-q^{-\alpha}}f(0)+\int_{\norm{x} > 1} \norm{x}^{\alpha -n} f(x) \d x
+\int_{\norm{x} \leq 1}\norm{x}^{\alpha -n} [f(x)-f(0)] \d x. 
\end{equation*}
If $\Re \alpha>0$ and if $f \in \scr{D}(\K^n)$, this formula can be rewritten (see \cite[Exercise 9.6 p.~34]{Zun22b} \cite[p.~27]{Zun16}) under the form
\begin{equation}
\label{action-k-alpha}
\big\la \norm{\cdot}_{\K^n}^{-\alpha-n}, f \big\ra
=\int_{\K^n-\{0\}} \big[f(x)-f(0)\big] \frac{\d x}{\norm{x}_{\K^n}^{\alpha+n}}.
\end{equation}
The Fourier transform of the distribution $\norm{\cdot}_{\K^n}^{\alpha-n}$ is known. Actually, by \cite[Theorem 2 p.~194]{Tai68}, \cite[Theorem 4.5 p.~130]{Tai75}, \cite[(4.3) p.~136]{VVZ94},  we have in $\scr{D}'(\K^n)$
\begin{equation}
\label{Fourier-transform-kalpha}
\cal{F}\big(\norm{\cdot}_{\K^n}^{\alpha-n}\big)
=\Gamma_q^{(n)}(\alpha)\norm{\cdot}_{\K^n}^{-\alpha}, \quad\alpha \in \mathbb{C}-(n+\tfrac{2\pi \i}{\ln q}\Z).
\end{equation}

Now, we can prove the following result.
\begin{prop}
\label{prop-norm-alpha-negative-definite}
Let $\K$ be a non-Archimedean local field and consider some integer $n \geq 1$. For any $0 < \alpha <\infty$, we have
\begin{equation}
\label{Levy-formula-p-adc}
\norm{x}_{\K^n}^\alpha
=\frac{-1}{\Gamma_q^{(n)}(-\alpha)} \int_{\K^n-\{0\}} \frac{1-\chi(x \cdot y) }{\norm{y}_{\K^n}^{\alpha+n}} \d y, \quad x \in \K^n,
\end{equation} 
In particular, the function $\K^n \to \R$, $x \mapsto \norm{x}_{\K^n}^\alpha$ is a conditionally negative definite function on the additive group of $\K^n$. 
\end{prop}

\begin{proof}
For any function $f \in \scr{D}(\K^n)$, using \eqref{Fourier-transform-kalpha} with $-\alpha$ instead of $\alpha$, we have
\begin{align*}
\MoveEqLeft
\big\la \norm{\cdot}_{\K^n}^\alpha,f \big\ra         
\ov{\eqref{Fourier-transform-kalpha}}{=} \Gamma_q^{(n)}(\alpha+n) \big\la \cal{F}^{-1}\big(\norm{\cdot}_{\K^n}^{-\alpha-n}\big), f \big\ra 
= \Gamma_q^{(n)}(\alpha+n) \big\la \norm{\cdot}_{\K^n}^{-\alpha-n}, \check{f} \big\ra \\
&\ov{\eqref{prod-Gamma}\eqref{action-k-alpha}}{=} \frac{1}{\Gamma_q^{(n)}(-\alpha)} \int_{\K^n - \{ 0 \}} \big[\check{f}(x)-\check{f}(0)\big] \frac{\d x}{\norm{x}_{\K^n}^{\alpha+n}} \\
&\ov{\eqref{inverse-Fourier-transform}}{=} \frac{1}{\Gamma_q^{(n)}(-\alpha)}  \int_{\K^n - \{ 0 \}} \bigg(\int_{\K^n} (\chi(x \cdot y)-1) f(y) \d y \bigg) \frac{\d x}{\norm{x}_{\K^n}^{\alpha+n}} \\
&=\frac{1}{\Gamma_q^{(n)}(-\alpha)} \sum_{k \in \Z} \int_{\norm{x}_{\K^n} = q^{-k}} \bigg(\int_{\K^n} (\chi(x \cdot y)-1) f(y) \d y\bigg) \frac{\d x}{\norm{x}_{\K^n}^{\alpha+n}}.
\end{align*}
Now, we aim to apply Fubini's theorem to interchange the order of integration. To do so, we must ensure that the double integral  
\begin{equation}
\label{rree-88}
\int_{\norm{x}_{\K^n} = q^{-k}}
\bigg(\int_{\K^n}
\bigl|\chi(x \cdot y) - 1\bigr|
|f(y)|
\frac{\d x}{\norm{x}_{\mathbb{K}^n}^{\alpha+n}}\bigg) \d y
\end{equation}
is finite. There exists an integer $N \geq 1$ such that $f(y)=0$ if $\norm{y}_{\K^n} \geq q^N$ and a positive constant $M$ such that $|f(y)| \leq M$ for any $y \in \K^n$. We get the inequality
\[
\bigl|\chi(x \cdot y) - 1\bigr|\,
|f(y)|\,
\frac{1}{\norm{x}_{\mathbb{K}^n}^{\alpha+n}}
\leq
\frac{2M}{\norm{x}_{\mathbb{K}^n}^{\alpha+n}},
\]
Putting these points together, we see that the previous integral \eqref{rree-88} is less than
\begin{equation*}
\label{}
2M \bigg(\int_{\norm{y}_{\K^n} \leq q^N}\d y\bigg) 
\bigg(\int_{\norm{x}_{\K^n} = q^{-k}} \frac{\d x}{\norm{x}_{\mathbb{K}^n}^{\alpha+n}}
\bigg)< \infty. 
\end{equation*}
Consequently, using Fubini's theorem, we can switch the integrals over $y$ and $x$ and get
\[
\big\la \norm{\cdot}_{\K^n}^\alpha,f \big\ra
= \frac{1}{\Gamma_q^{(n)}(-\alpha)} \sum_{k \in \Z} \int_{\mathbb{K}^n} \bigg(\int_{\norm{x}_{\mathbb K^n} = q^{-k}} (\chi(x \cdot y)-1) \frac{\d x}{\norm{x}_{\K^n}^{\alpha+n}}\bigg)  f(y) \d y 
\]
Using again the fact that $f$ is bounded and has compact support and the series over the integrals in $x$ converges absolutely according to Lemma \ref{lem-integrals-over-spheres}, we 
can again apply Fubini's theorem, this time to the series in $k$ and the integral over $y$ and obtain
\[ 
\big\la \norm{\cdot}_{\K^n}^\alpha,f \big\ra=\frac{1}{\Gamma_q^{(n)}(-\alpha)} \int_{\K^n } f(y) \bigg(\int_{\K^n- \{ 0 \}} \frac{\chi(x \cdot y)-1}{\norm{x}_{\K^n}^{\alpha+n}} \d x\bigg)\d y.
\]
\end{proof}


\begin{remark}\normalfont
\label{rem-imaginary-part-character}
Note that the formula \eqref{Levy-formula-p-adc} remains true if we replace $\chi(x \cdot y)$ by $\Re \chi(x \cdot y)$ since the improper integral $\int_{\K^n - \{ 0 \}} \frac{\Im \chi(x \cdot y)}{\norm{y}_{\K^n}^{\alpha + n}} \d y$ vanishes. Indeed, since $\chi$ is a character, we have $\chi(-x \cdot y) = \overline{\chi(x\cdot y)}$ for any $x,y \in \mathbb K^n$. We infer that $\Im \chi(-x \cdot y) = \Im \overline{\chi(x \cdot y)} = - \Im \chi (x \cdot y)$. Using the unimodularity of the \textit{abelian }additive group $\K^n$ in the first equality, we deduce that
\begin{align*}
\MoveEqLeft
\int_{\K^n - \{ 0 \}} \frac{\Im \chi(x \cdot y)}{\norm{y}_{\K^n}^{\alpha + n}} \d y 
= \int_{\K^n - \{ 0 \}} \frac{\Im \chi(-x \cdot y)}{\norm{-y}_{\K^n}^{\alpha + n}} \d y 
= -\int_{\K^n - \{ 0 \}} \frac{\Im \chi(x \cdot y)}{\norm{y}_{\K^n}^{\alpha + n}} \d y.
\end{align*}
Here, we make a slight abuse of notation since the integrals are improper.
\end{remark}

\begin{remark} \normalfont
If $n=1$ and $\K=\mathbb{Q}_q$, we recover the formula 
$$
|x|_q^\alpha
=\frac{q^\alpha-1}{1-q^{-1-\alpha}} \int_{\mathbb{Q}_q-\{0\}} \frac{1-\Re \chi(x y)}{|y|_q^{\alpha+1}} \d y, \quad x \in \mathbb{Q}_q.
$$
of the papers \cite[p.~699]{Har90} and \cite[p.~908]{Hara2}. By \eqref{character}, we have $\Re \chi(x y)=\cos 2\pi \{x y\}_{q}$. 
\end{remark}

As a consequence of Proposition \ref{prop-norm-alpha-negative-definite}, by \cite[Theorem 8.3 p.~49 and pp.~51]{BeF75}, there exists a convolution semigroup $(\mu_t)_{t \geq 0}$ of symmetric probability measures on $\K^n$ such that 
$$
\hat{\mu}_t
=\e^{-t\norm{\cdot}_{\K^n}^\alpha}, \quad t \geq 0.
$$
Note that the function $\xi \mapsto \e^{-t\norm{\xi}_{\K^n}^\alpha}$ belongs to the space $\L^1(\K^n)$ for any $t > 0$. Indeed, we have
\begin{align}
\MoveEqLeft
\label{fin-567}
\int_{\K^n} \e^{-t\norm{\xi}_{\K^n}^\alpha}   \d \xi      
\leq \int_{\norm{\xi}_{\K^n} \leq 1} \e^{- t \norm{\xi}_{\K^n}^\alpha} \d \xi + \int_{\norm{\xi}_{\K^n} > 1} \e^{-t \norm{\xi}_{\K^n}^\alpha} \d \xi \\
&\leq \int_{\norm{\xi} \leq 1} \d \xi + \sum_{k=1}^{\infty} \int_{\norm{\xi}_{\K^n}=q^k} \e^{-t \norm{\xi}_{\K^n}^\alpha} \d \xi \nonumber\\
&= 1 + \sum_{k=1}^{\infty} \e^{-t q^{k\alpha}} \int_{\norm{\xi}_{\K^n}=q^{k}} \d \xi \nonumber\\
&\ov{\eqref{int-sphere}}{=} 1 + \sum_{k=1}^\infty \e^{-t q^{k\alpha}} q^{kn}(1-q^{-n})
< \infty. \nonumber
\end{align}
We deduce by \cite[Theorem 2.6 p.~10]{BeF75} that the measure $\mu_t$ admits a continuous density $K_t$ with respect to the Haar measure $\mu$ and that
$$
K_t(x)
\ov{\mathrm{def}}{=} \int_{\K^n} \e^{-t\norm{\xi}_{\K^n}^\alpha} \ovl{\chi(x \cdot \xi)} \d \xi, \quad t > 0,\: x \in \K^n.
$$
So we have a weak* continuous semigroup $(T_t)_{t \geq 0}$ of selfadjoint contractive positive Fourier multipliers on the algebra $\L^\infty(\K^n)$. If $t>0$, each map $T_t \co \L^\infty(\K^n) \to \L^\infty(\K^n)$ is the convolution operator by the kernel $K_t$. More generally, we need the kernel for complex times
\begin{equation}
\label{Def-kz}
K_z(x)
\ov{\mathrm{def}}{=} \int_{\K^n} \e^{-z\norm{\xi}_{\K^n}^\alpha} \ovl{\chi(x \cdot \xi)} \d \xi, \quad \Re z > 0,\: x \in \K^n.
\end{equation}
The previous integral converges absolutely since
\begin{align*}
\MoveEqLeft
\int_{\K^n} \left|\e^{-z\norm{\xi}_{\K^n}^\alpha} \ovl{\chi(x \cdot \xi)}\right|  \d \xi       
\leq \int_{\K^n} \e^{- \Re z \norm{\xi}_{\K^n}^\alpha} \d \xi
\ov{\eqref{fin-567}}{<} \infty.
\end{align*}
Note that the generator of the Markovian semigroup $(T_t)_{t \geq 0}$ is the operator $-D^\alpha$, where $D^\alpha$ is the Taibleson operator, defined in \eqref{Taibleson-Fourier}.

%

The following is a slight extension of \cite[Lemma 2 p.~331]{RoZ08} (see also \cite[(4.5) p.~147]{Koc01}). 
Again, we refer to \cite{Ber52}, \cite[Lemma 1 p.~583]{Fel71} and \cite[Theorem 4.1]{Jan22} for the analogous result for $\R^n$.

\begin{prop}
\label{prop-noyau-en-serie}
Let $\alpha > 0$. For any complex number $z \in \mathbb{C}$ with $\Re z > 0$ and any $x \in \K^n-\{0\}$, we have
\begin{equation}
\label{Noyau-en-serie}
K_z(x)
= \sum_{k=1}^{\infty} \frac{(-z)^k}{k!} \frac{\Gamma_q^{(n)}(k\alpha + n)}{\norm{x}_{\K^n}^{k\alpha+n}}.
\end{equation}
\end{prop}

\begin{proof}
Let $x \in \K^n-\{0\}$. We have $\norm{x}_{\K^n}=q^{n_0}$ for some integer $n_0 \in \Z$. We expand as usual the integral over $\K^n$ into crowns and we get
\begin{align*}
\MoveEqLeft
K_z(x) 
\ov{\eqref{Def-kz}}{=} \int_{\K^n} \e^{-z \norm{\xi}_{\K^n}^\alpha} \ovl{\chi(x \cdot \xi)} \d \xi 
= \sum_{k \in \Z} \int_{\norm{\xi}_{\K^n} =q^{-k}} \e^{-z q^{-k\alpha}} \ovl{\chi(x \cdot \xi)} \d \xi \\
&= \sum_{k \in \Z} \e^{-z q^{-k\alpha}} \int_{\norm{\xi}_{\K^n} =q^{-k}} \ovl{\chi(x \cdot \xi)} \d \xi \\
&\ov{\eqref{equ-proof-prop-norm-alpha-negative-definite}}{=} (1-q^{-n})\bigg(\sum_{\norm{x}_{\K^n} \leq q^k} \e^{-z q^{-k\alpha}}q^{-kn}\bigg) -\e^{-z \frac{q^\alpha}{\norm{x}_{\K^n}^\alpha}} \norm{x}_{\mathbb K^n}^{-n} \\
&=(1-q^{-n})\bigg(\sum_{k \geq n_0} \e^{-z q^{-k\alpha}}q^{-kn}\bigg) -\e^{-z \frac{q^\alpha}{\norm{x}_{\K^n}^\alpha}} \norm{x}_{\mathbb K^n}^{-n} \\
&=(1-q^{-n})\bigg(\sum_{k \geq 0} \e^{-z q^{-k\alpha}\norm{x}_{\mathbb K^n}^{-\alpha}}q^{-kn}\norm{x}_{\K^n}^{-n}\bigg) -\e^{-z \frac{q^\alpha}{\norm{x}_{\K^n}^\alpha}} \norm{x}_{\mathbb K^n}^{-n}, 
\end{align*}
using a change of index in the last equality. We expand in the series over $k$ the exponential function into a series over $l$ and we obtain
\[
K_z(x) 
=\frac{1 - q^{-n}}{\norm{x}_{\K^n}^n}\bigg(\sum_{k \geq 0} \sum_{l \geq 0} \frac{(-z)^l}{l!} \left( \frac{q^{-k \alpha}}{\norm{x}_{\K^n}^\alpha} \right)^l q^{-kn} \bigg) - \e^{-z \frac{q^\alpha}{\norm{x}_{\K^n}^\alpha}} \norm{x}_{\K^n}^{-n}
. 
\]
This double series converges absolutely. Indeed, we can bound $\sum_{k \geq 0} \sum_{l \geq 0} |\cdots|$ by
\[ 
\sum_{k \geq 0} \exp\bigg(|z| \frac{q^{-k\alpha}}{\norm{x}_{\K^n}^\alpha}\bigg)  q^{-kn}
< \infty. 
\]
Thus we can interchange the order of summation. Expanding the second exponential in the first equality and summing a geometric series in the third equality, we deduce that
\begin{align*}
\MoveEqLeft
K_z(x) 
= (1-q^{-n})\sum_{l \geq 0} \frac{(-z)^l}{l!} \frac{1}{\norm{x}_{\K^n}^{l\alpha  + n}} \sum_{k \geq 0} q^{-l\alpha k-kn} - \sum_{l \geq 0} \frac{(-z)^l}{l!} \frac{q^{l\alpha}}{\norm{x}_{\K^n}^{l\alpha+n}} \\
&=\sum_{l \geq 0} \frac{(-z)^l}{l!}\frac{1}{\norm{x}_{\K^n}^{l\alpha  + n}}\bigg( (1-q^{-n})  \bigg(\sum_{k \geq 0} q^{-k(l\alpha+n)} \bigg) -  q^{l\alpha} \bigg)\\
&=\sum_{l \geq 0} \frac{(-z)^l}{l!}\frac{1}{\norm{x}_{\K^n}^{l\alpha  + n}}\bigg( \frac{1-q^{-n}}{1-q^{-(l\alpha+n)}} -  q^{l\alpha} \bigg)\\
&=\sum_{l \geq 0} \frac{(-z)^l}{l!}\frac{1}{\norm{x}_{\K^n}^{l\alpha  + n}}\bigg( \frac{1-q^{l\alpha}}{1-q^{-(l\alpha+n)}} \bigg)
\ov{\eqref{Gamma-formula-n}}{=} \sum_{l \geq 0} \frac{(-z)^l}{l!} \frac{1}{\norm{x}_{\K^n}^{l \alpha + n}} \Gamma_q^{(n)}(l \alpha + n).
\end{align*}
Finally, recall that $\Gamma_q^{(n)}(n) \ov{\eqref{Gamma-formula-n}}{=} 0$. So the first term is zero.
\end{proof}

%

The following is a generalization of \cite[Lemma 4.1 p.~147]{Koc01},  which gives for $n=1$ the estimate 
\begin{equation*}
\label{}
K_t(x) 
\lesssim \frac{t}{(t^{\frac{1}{\alpha}}+\norm{x}_{\mathbb{K}})^{\alpha+1}}
\end{equation*}
for any $t > 0$. See also \cite[Lemma 3 p.~332]{RoZ08} for the $n$-dimensional variant
\begin{equation*}
\label{}
K_t(x) 
\lesssim \frac{t}{(t^{\frac{1}{\alpha}}+\norm{x}_{\Q_q^n})^{\alpha+n}}.
\end{equation*}
We also refer to the papers \cite{BGPV14}, \cite{BGHH21} and to the survey \cite{Gri23} for heat kernel estimates in a larger ultrametric context.

\begin{lemma}
\label{Lemma-Estimation-noyau}
Let $\alpha > 0$. For any $x \in \K^n$ and any complex number such that $\Re z >0$, we have
\begin{equation}
\label{First-estimate-kernel}
|K_z(x)|
\lesssim_{q,n} \frac{|z|}{[(\Re z)^{\frac{1}{\alpha}}+\norm{x}_{\K^n}]^{\alpha+n}}.
\end{equation}
\end{lemma}

\begin{proof}
Let $k \in \Z$ be an integer such that 
\begin{equation}
\label{ine-sans-fin}
q^{k-1} \leq (\Re z)^{\frac{1}{\alpha}} \leq q^k.
\end{equation}
Let $\beta$ be a prime element of $\K$. We have $|\beta| = q^{-1}$ and consequently $q^{k-1}=|\beta|^{1-k}$. Then using the change of variable $\eta = \beta^{1-k}\xi$ in the third equality, we obtain
\begin{align}
\MoveEqLeft
\label{11}
|K_z(x)| 
\ov{\eqref{Def-kz}}{=} \left|\int_{\K^n} \e^{-z\norm{\xi}_{\K^n}^\alpha} \ovl{\chi(x \cdot \xi)}\d \xi \right| 
\leq \int_{\K^n}\left|\e^{-z\norm{\xi}_{\K^n}^\alpha}\ovl{\chi(x \cdot \xi)}\right| \d \xi \\
&\leq \int_{\K^n} \e^{-\Re(z) \norm{\xi}_{\K^n}^\alpha} \d \xi 
\ov{\eqref{ine-sans-fin}}{\leq} \int_{\K^n} \e^{-q^{\alpha(k-1)}\norm{\xi}_{\K^n}^\alpha} \d \xi \nonumber \\ 
&=\int_{\K^n} \e^{-|\beta|^{(1-k)\alpha}\norm{\xi}_{\K^n}^\alpha}  \d \xi 
=\int_{\K^n} \e^{-\norm{\beta^{1-k}\xi}_{\K^n}^\alpha} \d \xi  \nonumber \\
&\ov{\eta = \beta^{1-k}\xi}{=} \frac{1}{q^{(k-1)n}}\int_{\K^n} \e^{-\norm{\eta}_{\K^n}^\alpha} \d \eta \nonumber
=\frac{1}{q^{kn}} \bigg(q^{n} \int_{\K^n} \e^{-\norm{\eta}_{\K^n}^\alpha} \d \eta\bigg) 
\ov{\eqref{ine-sans-fin}}{\lesssim_{q,n}} \frac{1}{(\Re z)^\frac{n}{\alpha}}.
\end{align}
On the other hand, according to \eqref{Noyau-en-serie}, we have for any complex number $z$ with $\Re z>0$,
\begin{align*}
\MoveEqLeft
K_z(x) 
\ov{\eqref{Noyau-en-serie}}{=} \sum_{k=0}^\infty \frac{(-z)^k}{k!} \frac{\Gamma_q^{(n)}(k\alpha + n)}{\norm{x}_{\K^n}^{k \alpha + n}} 
\ov{\eqref{Gamma-formula-n}}{=} \sum_{k=0}^\infty \frac{(-z)^k}{k!} \frac{1-q^{k \alpha}}{1 - q^{-k\alpha - n}}\frac{1}{\norm{x}_{\K^n}^{k \alpha + n}} \\
& = \sum_{k=0}^\infty \frac{(-z)^k}{k!} \frac{1}{1 - q^{-k\alpha - n}}\frac{1}{\norm{x}_{\K^n}^{k \alpha + n}} - \sum_{k=0}^\infty \frac{(-z)^k}{k!} \frac{q^{k\alpha}}{1 - q^{-k\alpha - n}}\frac{1}{\norm{x}_{\K^n}^{k \alpha + n}}\\
& = \sum_{k=0}^\infty \frac{(-z)^k}{k!} \sum_{l=0}^\infty q^{-lk\alpha - ln}\frac{1}{\norm{x}_{\K^n}^{k \alpha + n}} - \sum_{k=0}^\infty \frac{(-z)^k}{k!} \sum_{l=0}^\infty q^{k\alpha} q^{-lk\alpha - ln}\frac{1}{\norm{x}_{\K^n}^{k \alpha + n}}\\
& = \sum_{l=0}^\infty q^{-ln} \norm{x}_{\K^n}^{-n} \sum_{k=0}^\infty \frac{(-z)^k}{k!} \frac{q^{-lk\alpha}}{\norm{x}_{\K^n}^{k \alpha}} - \sum_{l=0}^\infty q^{-ln} \norm{x}_{\K^n}^{-n} \sum_{k=0}^\infty \frac{(-z)^k}{k!} q^{k\alpha} q^{-lk\alpha}\frac{1}{\norm{x}_{\K^n}^{k \alpha}}\\
& = \sum_{l=0}^\infty q^{-ln} \norm{x}_{\K^n}^{-n} \left( \exp(-q^{-l\alpha}\norm{x}_{\K^n}^{-\alpha} z) - \exp(-q^{-l\alpha+\alpha}\norm{x}_{\K^n}^{-\alpha} z) \right).
\end{align*}
Note that the elements $q^{-l\alpha}\norm{x}_{\K^n}^{-\alpha} z$ and $q^{-l\alpha+\alpha}\norm{x}_{\K^n}^{-\alpha} z$ are complex numbers with real part $>0$. We infer by the mean value theorem\footnote{\thefootnote. Note that $\e^{-\Re \zeta} \leq 1$ if $\Re \zeta > 0$.} applied to the function $z \mapsto \e^{-z}$
\begin{align}
\MoveEqLeft
|K_z(x)| 
\leq \sum_{l=0}^\infty q^{-ln} \norm{x}_{\K^n}^{-n} \norm{x}_{\K^n}^{-\alpha} |z| \, |q^{-l\alpha}(1-q^\alpha)| 
\lesssim \frac{|z|}{\norm{x}_{\K^n}^{\alpha + n}} \label{equ-12}.
\end{align}
We will examine two cases.

\noindent\textit{First case.} If $\norm{x}_{\mathbb{K}^n} \leq (\Re z)^{\frac{1}{\alpha}}$, we have
\begin{align*}
\MoveEqLeft
|K_z(x)|
\ov{\eqref{11}}{\lesssim} \frac{1}{(\Re z)^\frac{n}{\alpha}}
= \frac{\Re z}{[(\Re z)^{\frac{1}{\alpha}}]^{\alpha+n}} 
\leq \frac{\Re z}{[\frac{1}{2}(\Re z)^{\frac{1}{\alpha}}+\frac{1}{2}\norm{x}_{\K^n}]^{\alpha+n}} 
\lesssim \frac{\Re z}{[(\Re z)^{\frac{1}{\alpha}}+\norm{x}_{\K^n}]^{\alpha+n}}.
\end{align*}
\textit{Second case.} Now, if $\norm{x}_{\K^n} \geq (\Re z)^{\frac{1}{\alpha}}$ then
\begin{align*}
\MoveEqLeft
|K_z(x)| \ov{\eqref{equ-12}}{\lesssim} \frac{|z|}{\norm{x}_{\K^n}^{\alpha+n}}
\leq \frac{|z|}{[\frac{1}{2}(\Re z)^{\frac{1}{\alpha}}+\frac{1}{2}\norm{x}_{\K^n}]^{\alpha+n}}
\lesssim \frac{|z|}{[(\Re z)^{\frac{1}{\alpha}}+\norm{x}_{\K^n}]^{\alpha+n}}.
\end{align*}
%
\end{proof}

We will use the following estimate. It is stated in \cite[Lemma 4.3 p.~151]{Koc01} that for any $\alpha > 0$ and any $b > 0$ we have
\begin{equation}
\label{Magic-estimate}
\int_{\K}\frac{1}{[b+\norm{x}_{\K}]^{\alpha+1}} \d x
\lesssim \frac{1}{b^\alpha}.
\end{equation}
The authors \cite[Proposition 2 p.~335]{RoZ08} provide an $n$-dimensional generalization, but only for $\mathbb{Q}_q^n$. However, the proof actually holds for the case $\K^n$ where $\K$ is a non-Archimedean local field. This leads to the following result. 

\begin{lemma}
\label{lem-Magic-estimate-n}
Let $n \in \N$, $\alpha > 0$ and $b > 0$.
We have
\begin{equation}
\label{Magic-estimate-n}
\int_{\K^n} \frac{1}{[b+\norm{x}_{\K^n}]^{\alpha+n}} \d x \lesssim \frac{1}{b^\alpha},
\end{equation}
where the implicit constant is independent of $b$.
\end{lemma}


The following result is an analogue of the classical estimate 
$$
\norm{h_z}_{\L^1(\R^n)} 
\leq \bigg(\frac{|z|}{\Re z}\bigg)^\frac{n}{2},
$$ 
provided by \cite[Remark 3.7.10 p.~153]{ABHN11} (or \cite[p.~543]{HvNVW18}) on the kernel $h_z$ of the heat semigroup on $\R^n$ for any complex number $z$ with $\Re z > 0$.

\begin{cor}
\label{cor-estim-int}
Let $\alpha >0$ and $n \in \N$. For any complex number $z$ such that $\Re z > 0$, we have
\begin{equation*}
\norm{K_z}_{\L^1(\mathbb{K}^n)}
\lesssim_{q,n} \frac{|z|}{\Re z}.
\end{equation*}
The same estimate also holds for $K_z$ replaced by its radially decreasing majorant $\cal{R}_{K_z}(x) \ov{\eqref{radially-majorant}}{=} \underset{|x| \leq |r|} \esssup\, |K_z(r)|$, i.e.
\begin{equation}
\label{majo-final}
\norm{\cal{R}_{K_z}}_{\L^1(\mathbb{K}^n)}
\lesssim_{q,n} \frac{|z|}{\Re z}.
\end{equation}
\end{cor}

\begin{proof}
For any complex number $z$ such that $\Re z >0$, we have
\begin{align*}
\MoveEqLeft
\norm{K_z}_{\L^1(\mathbb{K}^n)}
=\int_{\K^n} |K_z(x)| \d x            
\ov{\eqref{First-estimate-kernel}}{\lesssim_{q,n}} \int_{\K^n} \frac{|z|}{\big[(\Re z)^{\frac{1}{\alpha}}+\norm{x}_{\K^n}\big]^{\alpha+n}} \d x 
\ov{\eqref{Magic-estimate-n}}{\lesssim} \frac{|z|}{\Re z}.
\end{align*} 
We turn to the last sentence in the statement.
Note that according to Example \ref{Ex-radial-bis}, we have $|x| \ov{\eqref{compar-mod-2}}{=} \norm{x}_{\K^n}^n$ for any $x \in \K^n$. Thus, the estimate \eqref{First-estimate-kernel} bounds $K_z(x)$ by a radial and radially decreasing term in $|x|$ (as well as in $\norm{x}_{\K^n}$). We infer that \eqref{First-estimate-kernel} holds with the left-hand side $K_z$ replaced by $\mathcal{R}_{K_z}$, and the calculation hereabove finishes the proof.
\end{proof}

In the following theorem, we use the semigroup $(T_t)_{t \geq 0}$ of operators acting on the Banach space $\L^p(\K^n)$ given by convolution with the kernel $K_t(x) \ov{\eqref{Def-kz}}{=} \int_{\K^n} e^{-t \|\xi\|_{\K^n}^\alpha} \overline{\chi(x \cdot \xi)} \d \xi$, where $t>0$ and $x \in \K^n$. Its infinitesimal generator is (the closure of) the operator $-D^\alpha$, where $D^\alpha$ is the Taibleson operator, defined in \eqref{Taibleson-Fourier}. 

\begin{thm}
\label{Main-Th-angle-0-p-adic}
Let $\K$ be a non-Archimedean local field and $n \geq 1$. Consider a $\UMD$ Banach function space $Y$. Suppose that $1 < p < \infty$. The negative of the infinitesimal generator $-D^\alpha \ot \Id_Y$ of the strongly continuous semigroup $(T_t \ot \Id_Y)_{t \geq 0}$ of operators, i.e.~$D^\alpha \ot \Id_Y$, admits a bounded $\H^\infty(\Sigma_\theta)$ functional calculus on the Banach space $\L^p(\K^n,Y)$ for any angle $\theta > 0$ (i.e.~we have $\omega_{\H^\infty}=0$) and a bounded H\"ormander $\Hor^s_2(\R_+^*)$ calculus for any $s > \frac32$.
\end{thm}

\begin{proof}
First note that the Taibleson operator $D^\alpha \ot \Id_Y$ has dense range. Indeed, if $Y = \C$, the range of the operator $D^\alpha$ contains the Bruhat-Schwartz test function space $\scr{D}(\mathbb K^n)$ \cite[p.~117]{Tai75} which is dense in $\L^p(\mathbb K^n)$ for $1 < p < \infty$ \cite[p.~118]{Tai75}. Since $\L^p(\mathbb K^n) \ot Y$ is itself dense in $\L^p(\mathbb K^n,Y)$ and $D^\alpha \ot \Id_Y$ has dense range. Then as $(T_t)_{t \geq 0}$ is a semigroup of positive contractions, by Theorem \ref{Th-fun-dilation}, the operator $A$ admits a bounded $\H^\infty(\Sigma_\theta)$ functional calculus for any angle $\theta >\frac{\pi}{2}$.  Recall that $(T_z \ot \Id_Y)_{z \in \Sigma_{\theta}}$ denotes the complex time semigroup with generator $-D^\alpha \ot \Id_Y$. Note that each operator $T_z$ is given by convolution with the integral kernel $K_z$ from \eqref{Def-kz}. Let $\cal{R}_{K_z}(x) \ov{\eqref{radially-majorant}}{=} \underset{|x| \leq |r|}{\esssup}\, |K_z(r)|$ denote its radially decreasing majorant.

Fix an angle $\theta \in (0,\frac{\pi}{2})$. For any complex number $z$ with $|\arg z|  \leq \theta$, by means of the estimate provided by Corollary \ref{cor-estim-int}, we have
\begin{equation*}
\label{}
\norm{\cal{R}_{K_z}}_{\L^1(\mathbb{K}^n)}
\lesssim_{q,n} \frac{|z|}{\Re z}
=\frac{1}{\cos \arg z}
\leq \frac{1}{\cos \theta}.
\end{equation*}
According to Theorem \ref{Th-angle-scalar-Vilenkin}, we obtain that the family $(T_z)_{|\arg z| \leq \theta}$ is $R$-bounded over the Banach space $\L^p(\K^n)$ (in particular the semigroup $(T_t)_{t \geq 0}$ is bounded analytic). According to Theorem \ref{thm-R-anal-prime}, we deduce that $\omega_R(D^\alpha)=0$. By Theorem \ref{thm-R-anal}, we conclude that $\omega_{\H^\infty}(D^\alpha)=0$, i.e.~the operator $D^\alpha$ admits a bounded $\H^\infty(\Sigma_\theta)$ functional calculus for any angle $\theta > 0$.

For any complex number $z$ such that $\Re z >0$, using again the estimate provided by Corollary \ref{cor-estim-int}, we deduce that
\begin{equation*}
\label{}
\bnorm{|\cos \arg z|\cal{R}_{K_z}}_{\L^1(\mathbb{K}^n)}
\lesssim 1.
\end{equation*}
By Theorem \ref{Th-angle-scalar-Vilenkin}, we obtain that the family 
\begin{equation*}
\label{}
\left\{ (\cos \arg z) T_z :  \Re z > 0 \right\}
=\left\{ (\cos \arg z) (K_z * \cdot) :  \Re z > 0 \right\}
\end{equation*}
is $R$-bounded over the Banach space $\L^p(\K^n)$.
Since the semigroup $(T_t)_{t \geq 0}$ is bounded analytic, we conclude with Theorem \ref{prop-KrW18} that the operator $D^\alpha$ has a bounded H\"ormander $\Hor^s_2(\R_+^*)$ functional calculus for any $s > 1 + \frac12 = \frac32$. 

The vector-valued case is similar using Theorem \ref{Th-angle-UMD-lattice-Vilenkin} instead of Theorem \ref{Th-angle-scalar-Vilenkin}.
\end{proof}

\section{Application to local and global existence of solutions of semilinear evolution equations}
\label{sec-application-existence}

We finish by describing some applications to the well-posedness of some evolution equations. Here, we use the time weighted space $\L^{q}_\mu((0,a),Z)$ of the paper \cite[p.~2031]{PSW} defined by $y \in \L^{q}_\mu((0,a),Z)$ if and only if $t^{1 - \mu} y \in \L^q((0,a),Z)$ and similarly the vector-valued Sobolev space $\W^{1,q}_\mu((0,a),Z)$. The case $B = 0$ of the next theorem is already relevant. We note that one could also formulate variants of this result within the more general framework of \cite[Chapter~18]{HvNVW23}, leading to corresponding well-posedness results under alternative structural assumptions. Finally, we warn the reader that we follow the notation $X_{\gamma,\mu}$ of \cite{PSW}, which unfortunately differs from the notation used in \cite[p.~692]{HvNVW23} in the case $\gamma=1$.

\begin{thm}
\label{thm-evolution-equation}
Suppose that $1 < p,q < \infty$. Let $Y$ be a $\UMD$ Banach function space. Consider some $\alpha > 0$ and $n \in \N$.
Let $B$ be a sectorial operator acting on $Y$ admitting a bounded $\H^\infty(\Sigma_\omega)$ functional calculus of angle $< \frac{\pi}{2}$. Let $\beta \in [0,1)$ and write $X_\beta \ov{\mathrm{def}}{=} \dom(D^\alpha \ot \Id_Y + \Id_{\L^p(\K^n)} \ot B)^\beta$.
Consider a bounded bilinear map $G \co X_\beta \times X_\beta \to X_0$ together with the following semilinear parabolic evolution equation.
\begin{equation}
\label{equ-thm-evolution-equation}
\frac{\partial y}{\partial t}(t,s,x) + D^\alpha_s y(t,s,x) + B_x y(t,s,x) = G(y,y), \quad y(0) = y_0 .
\end{equation}
Then if the parameters $q,\mu$ and $\beta$ are chosen such that
$\mu \in (\frac{1}{q},1]$, $\beta \in (\mu - \frac{1}{q},1)$ and $2 \beta - 1 \leq \mu - \frac{1}{q}$, then for each initial value $y_0$ belonging to the real interpolation space $X_{\gamma,\mu} \ov{\mathrm{def}}{=} (X_0,X_1)_{\mu-\frac{1}{q},q}$, there exist a positive real number $a  = a(y_0)$ and a unique solution of the equation \eqref{equ-thm-evolution-equation} in the class
\[ 
y \in \W^{1,q}_\mu((0,a),X_0) \cap \L^{q}_\mu((0,a),X_1) \hookrightarrow \mathrm{C}([0,a],X_{\gamma,\mu}) .
\]
The solution exists on a maximal time interval $[0,t_+(y_0))$ and depends continuously on the data and satisfies
\[ 
y \in \W^{1,q}_\loc((0,t_+(y_0)),X_0) \cap \L^{q}
_\loc((0,t_+),X_1) \hookrightarrow \mathrm{C}((0,t_+(y_0)),X_{\gamma,1}).
\]
Hence it regularizes instantly, provided $\mu <  1$.
\end{thm}

\begin{proof}
According to Theorem \ref{Main-intro}, the operator $D^\alpha \ot \Id_Y$ admits a bounded $\H^\infty(\Sigma_\omega)$ functional calculus for any angle $\omega > 0$. Since $X = \L^p(\K^n,Y)$ is a Banach lattice, according to \cite[Theorem 1.1 p.~353]{LM03} (see also \cite[Theorem 1.2]{LLLM}), the operator $D^\alpha \ot \Id_Y + \Id_{\L^p(\K^n)} \ot B$ admits a bounded $\H^\infty(\Sigma_\omega)$ functional calculus for some angle $\omega < \frac{\pi}{2}$, in particular bounded imaginary powers to the same angle by \cite[Theorem 15.3.20 p.~476]{HvNVW23}. Since $X$ is a $\UMD$ Banach space, \cite[Theorem 2.1 p.~2033]{PSW} can be applied with the operator $A=D^\alpha \ot \Id_Y + \Id_{\L^p(\K^n)} \ot B$, $X_0=\L^p(\K^n,Y)$, $X_1=\dom A$ and with $q$ instead of $p$. This completes the proof.
\end{proof}

\begin{remark} \normalfont
In Theorem \ref{thm-evolution-equation}, already the case $B = 0$ is relevant in physical examples. Indeed, then \eqref{equ-thm-evolution-equation} is a semilinear parabolic heat equation that can model a $q$-adic reaction-diffusion equation with $t$ the time variable, $s$ the space variable and $x$, depending on the choice of the $\UMD$ Banach function space $Y$, a variable of population, velocity, trait or chemical state.
\end{remark}

\begin{remark} \normalfont
For relevant non-zero examples of operators $B$ in Theorem \ref{thm-evolution-equation} admitting a bounded $\H^\infty$ functional calculus, we refer to \cite[p.~262 -- 267]{KuW04}. This covers some elliptic operators on the space $\L^p(\R^n)$, either with constant coefficients or with H\"older continuous coefficients in the principal part, as well as divergence form operators with H\"older continuous coefficients.
\end{remark}

We further cite the following consequences from the results \cite[Corollary 2.2 p.~2034]{PSW} and \cite[Corollary 2.3 p.~2035]{PSW}.

\begin{cor}
Consider that the assumptions of Theorem \ref{thm-evolution-equation} hold.
\begin{enumerate}
\item
Then for any given $a > 0$ there exists $r = r(a) > 0$ such that the solution of the equation \eqref{equ-thm-evolution-equation} exists on $[0,a]$ whenever $\norm{y_0}_{X_{\gamma,\mu}} \leq r$.
\item
If $0 \in \rho(D^\alpha \ot \Id_Y + \Id_{\L^p(\K^n)} \ot B)$, then $r > 0$ is independent of $a$.
\item
If $0 \in \rho(D^\alpha \ot \Id_Y + \Id_{\L^p(\K^n)} \ot B)$ and $\frac{1}{q} < 1 - \beta$, then the trivial solution of the equation \eqref{equ-thm-evolution-equation} is exponentially stable in the state space $X_{\gamma,1}$.
Moreover, there exists $r_0 > 0$ such that the solution $y(t)$ of \eqref{equ-thm-evolution-equation} converges exponentially to zero in $X_{\gamma,1}$, provided $\norm{y_0}_{X_{\gamma,\mu}} \leq r_0$.
\end{enumerate}
\end{cor}

Concerning conditional global existence, the following holds.

\begin{cor}
The local solution of Theorem \ref{thm-evolution-equation} exists globally in each of the following cases:
\begin{enumerate}
\item $y([0,t_+(y_0))) \subset X_{\gamma,\mu}$ is bounded in the subcritical case $\mu - \frac1q > 2\beta - 1$,
\item $y([0,t_+(y_0))) \subset X_{\gamma,\mu}$ is relatively compact in the critical case $\mu - \frac1q = 2\beta - 1$.
\end{enumerate}
\end{cor}

We finish by giving a concrete example.

\begin{example} \normalfont
Suppose that $1 < p,q < \infty$. Let $\beta \in [0,1)$. Consider the Laplacian $B \ov{\mathrm{def}}{=} -\Delta$ on the space $Y=\L^q(\R^d)$ with domain $\dom B=\W^{2,q}(\R^d)$ by \cite[Remark 8.3.6 p.~226]{Haa06}. We have $X_0 =\L^p(\K^n,\L^q(\R^d))$ and $
\dom B^\beta=\W^{2\beta,q}(\R^d)$ according to \cite[Theorem 15.3.11 p.~462]{HvNVW23}. Set $A \ov{\mathrm{def}}{=}D^\alpha \ot \Id_Y + \Id_{\L^p(\K^n)} \ot B$, 
$$
A_s \ov{\mathrm{def}}{=} D^\alpha \ot \Id_{\L^q(\R^d)}
\quad \text{and} \quad
B_x \ov{\mathrm{def}}{=} \Id_{\L^p(\K^n)} \ot B.
$$
We have $A \ov{\mathrm{def}}{=} A_s + B_x$. Note that by \cite[Theorem 15.3.9 p.~460]{HvNVW23} and \cite[Theorem 15.4.11 p.~496]{HvNVW23}, we have 
\begin{align}
\MoveEqLeft
\label{embedding}
X_\beta
=\dom A^\beta 
=(X_0,\dom A )_{\beta}
=(X_0,\dom(A_s) \cap \dom(B_x) )_{\beta} \\      
&\subset (X_0,\dom A_s)_{\beta} \cap (X_0,\dom B_x)_{\beta}
= \dom A_s^\beta \cap \dom B_x^\beta. \nonumber
\end{align}
On the one hand, if $\beta > \frac{d}{2q}$, by the Sobolev embedding theorem \cite[Theorem 1.3.5 (c) p.~24]{Gra14b}, we have a continuous embedding
\[
\dom B^\beta
= \W^{2\beta,q}(\R^d)
\hookrightarrow
\L^\infty(\R^d).
\]
Using \cite[Theorem (1) p.~80]{DeF93} for tensorizing this embedding by $\Id_{\L^p(\K^n)}$ yields , this yields
\[
\dom B_x^\beta
= \L^p(\K^n,\dom B^\beta)
\hookrightarrow
\L^p(\K^n,\L^\infty(\R^d)).
\]
On the other hand, assume that the Sobolev embedding theorem \cite[Lemma 3 p.~199]{Tai68} admits the following vector-valued generalization with values in $\L^q(\R^d)$: if $\beta > \frac{n}{p\alpha}$, then
\[
\dom A_s^\beta
\hookrightarrow
\L^\infty(\K^n,\L^q(\R^d)).
\]
We expect that this can be proved by adapting Taibleson's argument to $\L^q(\R^d)$-valued functions, but including the details would take us too far afield. If $
\beta > \max\big\{\frac{d}{2q},\frac{n}{p\alpha}\big\}$, 
we obtain with \eqref{embedding} the continuous embeddings $
X_\beta \hookrightarrow \L^\infty(\K^n,\L^q(\R^d))$ and $X_\beta \hookrightarrow \L^p(\K^n,\L^\infty(\R^d))$.
Now, we can define the bilinear map $G \co X_\beta \times X_\beta \to X_0$ by 
\[
G(u,v)(s,x) 
\ov{\mathrm{def}}{=} u(s,x)v(s,x),
\qquad s \in \K^n, x \in \R^d,
\]
for any vector-valued functions $u,v \in X_\beta$. Then $G$ would be a bounded bilinear map since we have for any $u,v \in X_\beta$,
\begin{align*}
\norm{G(u,v)}_{X_0}
&= \norm{uv}_{\L^p(\K^n,\L^q(\R^d))} 
\leq \norm{u}_{\L^\infty(\K^n,\L^q(\R^d))}
\norm{v}_{\L^p(\K^n,\L^\infty(\R^d))} 
\lesssim \norm{u}_{X_\beta}\norm{v}_{X_\beta}.
\end{align*}
In particular, the semilinear equation in Theorem~\ref{thm-evolution-equation} covers the quadratic reaction-diffusion model $\partial_t y + D^\alpha y-\Delta y = y^2$.  
\end{example}

\paragraph{Declaration of interest} None.

\paragraph{Competing interests} The authors declare that they have no competing interests.

\paragraph{Data availability} No data sets were generated during this study.

\paragraph{Acknowledgements and Grants} We are thankful to Wilson Zuniga Galindo for short discussions. We thank the anonymous referee for a careful reading and for constructive comments that improved the presentation of the paper.



{\footnotesize

\vspace{0.2cm}

\noindent C\'edric Arhancet\\
6 rue Didier Daurat \\
81000 ALBI \\
FRANCE\\
URL: \href{http://sites.google.com/site/cedricarhancet}{https://sites.google.com/site/cedricarhancet}\\
cedric.arhancet@protonmail.com\\
ORCID: 0000-0002-5179-6972 

\vspace{0.2cm}

\noindent Christoph Kriegler\\
Universit\'e Clermont Auvergne\\
CNRS\\
LMBP\\
F-63000 CLERMONT-FERRAND\\
FRANCE \\
URL: \href{https://lmbp.uca.fr/~kriegler/indexenglish.html}{https://lmbp.uca.fr/{\raise.17ex\hbox{$\scriptstyle\sim$}}\hspace{-0.1cm} kriegler/indexenglish.html}\\
christoph.kriegler@uca.fr \\
ORCID: 0000-0001-8120-6251
}


\small
\begin{thebibliography}{79}

\bibitem[AbA02]{AbA02}
Y. A. Abramovich and C. D. Aliprantis.
\newblock An invitation to operator theory.
\newblock Graduate Studies in Mathematics, 50. American Mathematical Society, Providence, RI, 2002.


\bibitem[AVDR81]{AVDR81}
G. N. Agaev, N. Y. Vilenkin, G. M. Dzhafarli, A. I. Rubinshtein.
\newblock Multiplicative systems of functions and harmonic analysis on zero-dimensional groups (Russian).
\newblock Elm, Baku, 1981.

















\bibitem[AKS10]{AKS10}
S. Albeverio, A. Y. Khrennikov and V. M. Shelkovich.
\newblock Theory of $p$-adic distributions: linear and nonlinear models.
\newblock London Math. Soc. Lecture Note Ser., 370. Cambridge University Press, Cambridge, 2010.

\bibitem[Ale94]{Ale94}
G. Alexopoulos.
\newblock Spectral multipliers on Lie groups of polynomial growth.
\newblock Proc. Amer. Math. Soc. 120 (1994), no. 3, 973--979.









\bibitem[Are04]{Are04}
W. Arendt.
\newblock Semigroups and evolution equations: functional calculus, regularity and kernel estimates. 
\newblock Evolutionary equations. Vol. I, 1--85, Handb. Differ. Equ., North-Holland, Amsterdam, 2004. 


\bibitem[ABHN11]{ABHN11}
W. Arendt, C. J. K. Batty, M. Hieber and F. Neubrander.
\newblock Vector-valued Laplace transforms and Cauchy problems. Second edition.
\newblock Monographs in Mathematics, 96. Birkhäuser/Springer Basel AG, Basel, 2011.

\bibitem[Arh12]{Arh12}
C. Arhancet.
\newblock Unconditionality, Fourier multipliers and Schur multipliers.
\newblock Colloq. Math. 127 (2012), no. 1, 17--37.











 

\bibitem[ArK23]{ArK23}
C. Arhancet and C. Kriegler.
\newblock Projections, multipliers and decomposable maps on noncommutative $\L^p$-spaces.
\newblock M\'emoires de la Soci\'et\'e Math\'ematique de France (N.S.) (2023), no. 177.

 


\bibitem[AHKP25]{AHKP25}
C. Arhancet, L. Hagedorn, C. Kriegler and P. Portal.
\newblock The harmonic oscillator on the Moyal-Groenewold plane: an approach via Lie groups and twisted Weyl tuples. 
\newblock Preprint, 67 pages. \href{https://arxiv.org/pdf/2312.06143}{arXiv:2312.06143}














 


 










\bibitem[ABK99]{ABK99}
V. A. Avetisov, A. H. Bikulov and S. V. Kozyrev.
\newblock Application of $p$-adic analysis to models of breaking of replica symmetry.
\newblock J. Phys. A 32 (1999), no. 50, 8785--8791.

\bibitem[ABKO02]{ABKO02}
V. A. Avetisov, A. H. Bikulov, S. V. Kozyrev and V. A. Osipov.
\newblock $p$-adic models of ultrametric diffusion constrained by hierarchical energy landscapes.
\newblock J. Phys. A 35 (2002), no. 2, 177--189.

\bibitem[ABO03]{ABO03}
V. A. Avetisov, A. H. Bikulov and V. A. Osipov.
\newblock $p$-adic description of characteristic relaxation in complex systems.
\newblock J. Phys. A 36 (2003), no. 15, 4239--4246.

\bibitem[ABO04]{ABO04}
V. A. Avetisov, A. H. Bikulov and V. A. Osipov.
\newblock $p$-adic models of ultrametric diffusion in the conformational dynamics of macromolecules.
\newblock Tr. Mat. Inst. Steklova 245 (2004), 55--64. Proc. Steklov Inst. Math.(2004), no. 2, 48--57.

\bibitem[ABZ14]{ABZ14}
V. A. Avetisov, A. K. Bikulov and A. P. Zubarev.
\newblock Ultrametric random walk and dynamics of protein molecules.
\newblock Proc. Steklov Inst. Math. 285, 3--25 (2014).


\bibitem[BHV08]{BHV08}
B. Bekka, P. de la Harpe and A. Valette.
\newblock Kazhdan's property (T).
\newblock New Mathematical Monographs, 11. Cambridge University Press, Cambridge, 2008.

\bibitem[BGPV14]{BGPV14}
A. D. Bendikov, A. Grigor'yan, K. Pittet and V. Voess.
\newblock Isotropic Markov semigroups on ultra-metric space. 
\newblock Uspekhi Mat. Nauk 69 (2014), no. 4 (418), 3--102; translation in Russian Math. Surveys 69 (2014), no. 4, 589--680. 


\bibitem[BGHH21]{BGHH21}
A. Bendikov, A. Grigor'yan, E. Hu and J. Hu.
\newblock Heat kernels and non-local Dirichlet forms on ultrametric spaces.
\newblock Ann. Sc. Norm. Super. Pisa Cl. Sci. (5) 22 (2021), no. 1, 399--461.

\bibitem[BeS88]{BeS88}
C. Bennett and R. Sharpley.
\newblock Interpolation of operators.
\newblock Pure Appl. Math., 129. Academic Press, Inc., Boston, MA, 1988.




\bibitem[BeF75]{BeF75}
C. Berg and G. Forst.
\newblock Potential theory on locally compact abelian groups.
\newblock Ergebnisse der Mathematik und ihrer Grenzgebiete, Band 87. Springer-Verlag, New York-Heidelberg, 1975.

\bibitem[Ber52]{Ber52}
H. Bergstr\"om.
\newblock On some expansions of stable distribution functions.
\newblock Ark. Mat. 2 (1952), 375--378.












\bibitem[BiZ23]{BiZ23}
A. K. Bikulov and A. P. Zubarev.
\newblock Oscillations in $p$-adic diffusion processes and simulation of the conformational dynamics of protein.
\newblock $p$-Adic Numbers Ultrametric Anal. Appl. 15 (2023), no. 3, 169--186.


\bibitem[Blu03]{Blu03}%
S. Blunck.
\newblock A H\"ormander-type spectral multiplier theorem for operators without heat kernel.
\newblock Ann. Sc. Norm. Super. Pisa Cl. Sci. (5) 2 (2003), no. 3, 449--459.

\bibitem[Bou19]{Bou19}
N. Bourbaki.
\newblock \'El\'ements de math\'ematique. Th\'eories spectrales. Chapitres 1 et 2. (French) [Elements of mathematics. Spectral theories. Chapters 1 and 2]. Second edition.
\newblock Springer, Cham, 2019.

\bibitem[Bou98]{Bou98}
N. Bourbaki.
\newblock General topology. Chapters 1--4. Translated from the French. Reprint of the 1989 English translation. Elements of Mathematics.
\newblock (Berlin). Springer-Verlag, Berlin, 1998.

\bibitem[Bou04]{Bou04}
N. Bourbaki.
\newblock Integration. II. Chapters 7--9. Translated from the 1963 and 1969 French originals by Sterling K. Berberian. Elements of Mathematics (Berlin).
\newblock Springer-Verlag, Berlin, 2004.

\bibitem[Bou84]{Bou84}
J. Bourgain.
\newblock Extension of a result of Benedek, Calderon and Panzone.
\newblock Ark. Mat. 22 (1984), no. 1, 91--95.





\bibitem[BrF93]{BrF93}
L. Brekke and P. Freund.
\newblock $p$-adic numbers in physics.
\newblock Phys. Rep. 233 (1993), no. 1, 1--66.























\bibitem[Chr91]{Chr91}%
M. Christ.
\newblock $L^p$ bounds for spectral multipliers on nilpotent groups.
\newblock Trans. Amer. Math. Soc. 328 (1991), no. 1, 73--81.










%


\bibitem[CCMS17]{CCMS17}
V. Casarino, M. G. Cowling, A. Martini and A. Sikora.
\newblock Spectral multipliers for the Kohn Laplacian on forms on the sphere in $\mathbb{C}^n$.
\newblock J. Geom. Anal. 27 (2017), no. 4, 3302--3338.


\bibitem[DH16]{DH16}
Y. de Cornulier and P. de la Harpe.
\newblock Metric geometry of locally compact groups.
\newblock EMS Tracts in Mathematics 25, 2016.









%












\bibitem[CGHM94]{CGHM94}
M. Cowling, S. Giulini, A. Hulanicki and G. Mauceri.
\newblock Spectral multipliers for a distinguished Laplacian on certain groups of exponential growth.
\newblock Studia Math. 111 (1994), no. 2, 103--121.

\bibitem[CoS01]{CoS01}
M. Cowling and A. Sikora.
\newblock A spectral multiplier theorem for a sublaplacian on $\mathrm{SU}(2)$.
\newblock Math. Z. 238 (2001), no. 1, 1--36.

























\bibitem[DeF93]{DeF93}
A. Defant and K. Floret.
\newblock Tensor norms and operator ideals.
\newblock North-Holland Mathematics Studies, 176. North-Holland Publishing Co., Amsterdam, 1993.


\bibitem[DeE14]{DeE14}
A. Deitmar and S. Echterhoff.
\newblock Principles of harmonic analysis. Second edition.
\newblock Universitext Springer, Cham, 2014.

\bibitem[DeK17]{DeK17}
L. Deleaval and C. Kriegler.
\newblock Dunkl spectral multipliers with values in UMD lattices.
\newblock J. Funct. Anal. 272 (2017), no. 5, 2132--2175.

\bibitem[DKK21]{DKK21}
L. Deleaval, M. Kemppainen and C. Kriegler.
\newblock H\"ormander functional calculus on UMD lattice valued $L^p$ spaces under generalized Gaussian estimates.
\newblock J. Anal. Math. 145 (2021), no. 1, 177--234.

\bibitem[DeK23]{DeK23}
L. Deleaval and C. Kriegler.
\newblock Maximal H\"ormander functional calculus on $L^p$ spaces and UMD lattices.
\newblock Int. Math. Res. Not. IMRN (2023), no. 6, 4643--4694.

\bibitem[DeV22]{DeV22}
J. Delgado and J. P. Velasquez-Rodriguez.
\newblock Fundamental solution of the Vladimirov-Taibleson operator on noncommutative Vilenkin groups.
\newblock Preprint, arXiv:2204.07133.

 


\bibitem[Den21]{Den21}
R. Denk.
\newblock An introduction to maximal regularity for parabolic evolution equations.
\newblock Springer, Proc. Math. Stat., 346, Singapore, 2021, 1--70.








\bibitem[DJT95]{DJT95}
J. Diestel, H. Jarchow and A. Tonge.
\newblock Absolutely summing operators.
\newblock Cambridge Studies in Advanced Mathematics, 43. Cambridge University Press (1995).

\bibitem[Dieu78]{Dieu78}
J. Dieudonn\'e.
\newblock Treatise on Analysis, Vol. 6.
\newblock Academic Press, New York, 1978.







\bibitem[DHW24]{DHW24}
F. Ding, G. Hong and X. Wang.
\newblock A noncommutative maximal inequality for Fejér means on totally disconnected non-abelian groups.
\newblock Preprint, arXiv:2403.09263.


\bibitem[DKKVZ17]{DKKVZ17}
B. Dragovich, A. Y. Khrennikov, S. V. Kozyrev, I. V. Volovich and E. I. Zelenov.
\newblock $p$-adic mathematical physics: the first 30 years.
\newblock $p$-Adic Numbers Ultrametric Anal. Appl. 9 (2017), no. 2, 87--121.

\bibitem[Duo96]{Duo96}
X. T. Duong.
\newblock From the $L^1$ norms of the complex heat kernels to a H\"ormander multiplier theorem for sub-Laplacians on nilpotent Lie groups.
\newblock Pacific J. Math. 173 (1996), no. 2, 413--424.

\bibitem[DOS02]{DOS02}
X. T. Duong, E. M. Ouhabaz and A. Sikora.
\newblock Plancherel-type estimates and sharp spectral multipliers.
\newblock J. Funct. Anal. 196 (2002), no. 2, 443--485.


\bibitem[EdG77]{EdG77}
R. E. Edwards and G. I. Gaudry.
\newblock Littlewood-Paley and multiplier theory.
\newblock Ergebnisse der Mathematik und ihrer Grenzgebiete, Band 90. Springer-Verlag, Berlin-New York, 1977.







\bibitem[En89]{Eng89}
R. Engelking.
\newblock General topology.
\newblock Translated from the Polish by the author. Second edition. Sigma Series in Pure Mathematics, 6. Heldermann Verlag, Berlin, 1989.
























\bibitem[Fel71]{Fel71}
W. Feller.
\newblock An introduction to probability theory and its applications. Vol. II.
\newblock John Wiley \& Sons, Inc., New York-London-Sydney, 1971.

\bibitem[Fol16]{Fol16}
G. Folland.
\newblock A course in abstract harmonic analysis.
\newblock Second edition. Textbooks in Mathematics. CRC Press, Boca Raton, FL, 2016.















\bibitem[FGZ22]{FGZ22}
A. R. Fuquen-Tibatá, H. Garc\'ia-Compe\'an, W. A. Z\'uniga-Galindo.
\newblock Euclidean quantum field formulation of $p$-adic open string amplitudes.
\newblock Nuclear Phys. B 975 (2022), Paper No. 115684, 27 pp.

\bibitem[Gra14a]{Gra14a}
L. Grafakos.
\newblock Classical Fourier analysis. Third edition.
\newblock Graduate Texts in Mathematics, 249. Springer, New York, 2014.

\bibitem[Gra14b]{Gra14b}
L. Grafakos.
\newblock Modern Fourier analysis. Third edition.
\newblock Graduate Texts in Mathematics, 250. Springer, New York, 2014.

		
		
		
	
	

\bibitem[Gou20]{Gou20}
F. Q. Gouv\^ea.
\newblock $p$-adic numbers.
\newblock Universitext, Springer, Cham, 2020.


\bibitem[Gri23]{Gri23}
A. Grigor'yan.
\newblock Analysis on ultra-metric spaces via heat kernels.
\newblock $p$-Adic Numbers Ultrametric Anal. Appl. 15 (2023), no. 3, 204--242.



 








%












 





\bibitem[Haa06]{Haa06}
M. Haase.
\newblock The functional calculus for sectorial operators.
\newblock Operator Theory: Advances and Applications, 169. Birkh\"auser Verlag (2006).

\bibitem[HaP23]{HaP23}
M. Haase and F. Pannasch.
\newblock Holomorphic H\"ormander-type functional calculus on sectors and strips.
\newblock Trans. Amer. Math. Soc. Ser. B 10 (2023), 1356--1410.

\bibitem[Hal50]{Hal50}
M. Jr. Hall.
\newblock A topology for free groups and related groups.
\newblock Ann. of Math. (2) 52 (1950), 127--139.

\bibitem[Har90]{Har90}
S. Haran.
\newblock Riesz potentials and explicit sums in arithmetic.
\newblock Invent. Math. 101 (1990), no. 3, 697--703.

\bibitem[Hara93]{Hara2}
S. Haran.
\newblock Analytic potential theory over the $p$-adics.
\newblock Ann. Inst. Fourier (Grenoble) 43 (1993), no. 4, 905--944.








\bibitem[Heb90]{Heb90}
W. Hebisch.
\newblock A multiplier theorem for Schr\"odinger operators.
\newblock Colloq. Math. 60/61 (1990), no. 2, 659--664.

\bibitem[Heb93]{Heb93}
W. Hebisch.
\newblock Multiplier theorem on generalized Heisenberg groups.
\newblock Colloq. Math. 65 (1993), no. 2, 231--239.


\bibitem[HeR79]{HeR79}
E. Hewitt and K. A. Ross.
\newblock Abstract harmonic analysis. Vol. I.
Structure of topological groups, integration theory, group
representations. Second edition.
\newblock  Grundlehren der Mathematischen Wissenschaften, 115. Springer-Verlag, Berlin-New York, 1979.







\bibitem[HoM13]{HoM13}
K. H Hofmann and S. A. Morris.
\newblock The structure of compact groups. 3rd Edition, Revised and Augmented.
\newblock De Gruyter Studies in Mathematics 25, 2013. 























\bibitem[HvNVW16]{HvNVW16}
T. Hyt\"onen, J. van Neerven, M. Veraar and L. Weis.
\newblock Analysis in Banach spaces, Volume~I: Martingales and Littlewood-Paley theory.
\newblock Springer, 2016. 

\bibitem[HvNVW18]{HvNVW18}
T. Hyt\"onen, J. van Neerven, M. Veraar and L. Weis.
\newblock Analysis in Banach spaces, Volume~II: Probabilistic Methods and Operator Theory. 
\newblock Springer, 2018. 

\bibitem[HvNVW23]{HvNVW23}
T. Hyt\"onen, J. van Neerven, M. Veraar and L. Weis.
\newblock Analysis in Banach spaces, Volume~III: Harmonic Analysis and Spectral Theory. 
\newblock Springer, 2023. 










\bibitem[Jan22]{Jan22}
S. Janson.
\newblock Stable distributions.
\newblock Preprint, arXiv:1112.0220.














































\bibitem[KaW01]{KaW01}
N. J. Kalton and L. Weis.
\newblock The $\H^\infty$-calculus and sums of closed operators.
\newblock Math. Ann. 321 (2001), no. 2, 319--345.










\bibitem[KaV22]{KaV22}
A. Kassymov and J.P. Velasquez-Rodriguez.
\newblock Hardy inequalities on constant-order noncommutative Vilenkin groups.
\newblock Preprint, arXiv:2208.03495.




\bibitem[KKZ18]{KKZ18}
A. Y. Khrennikov, S. V. Kozyrev and W. A. Zuniga-Galindo.
\newblock Ultrametric pseudodifferential equations and applications. 
\newblock Encyclopedia of Mathematics and its Applications, 168. Cambridge University Press, Cambridge, 2018. 

\bibitem[KOJ16a]{KOJ16a}
A. Y. Khrennikov, K. Oleschko, M. de Jes\'us Correa L\'opez.
\newblock Application of $p$-adic wavelets to model reaction-diffusion dynamics in random porous media.
\newblock J. Fourier Anal. Appl. 22 (2016), no. 4, 809--822.

\bibitem[KOJ16b]{KOJ16b}
A. Khrennikov, K. Oleschko and M. de Jes\'us Correa L\'opez.
\newblock Modeling fluid's dynamics with master equations in ultrametric spaces representing the treelike structure of capillary networks.
\newblock Entropy 18 (2016), no. 7, Paper No. 249, 28 pp.

\bibitem[KhK20]{KhK20}
A. Y. Khrennikov and A. N. Kochubei.
\newblock On the $p$-adic Navier-Stokes equation.
\newblock Appl. Anal. 99 (2020), no. 8, 1425--1435.



 





















\bibitem[Koc01]{Koc01}
A. N. Kochubei.
\newblock Pseudo-differential equations and stochastics over non-Archimedean fields.
\newblock Monographs and Textbooks in Pure and Applied Mathematics, vol. 244, Marcel Dekker Inc., New York, 2001.






 






\bibitem[KrW18]{KrW18}
C. Kriegler and L. Weis.
\newblock Spectral multiplier theorems via $\H^\infty$ calculus and $R$-bounds.
\newblock Math. Z. 289 (2018), no. 1-2, 405--444.

\bibitem[KuU15]{KuU15}%
P. C. Kunstmann and M. Uhl.
\newblock Spectral multiplier theorems of H\"ormander type on Hardy and Lebesgue spaces.
\newblock J. Operator Theory 73 (2015), no. 1, 27--69.

\bibitem[KuW04]{KuW04}%
P. C. Kunstmann and L. Weis.
\newblock Maximal $L_p$-regularity for parabolic equations, Fourier multiplier theorems and $H^\infty$-functional calculus.
\newblock Functional analytic methods for evolution equations, 65--311, 
Lecture Notes in Math., 1855, Springer, Berlin, 2004.







\bibitem[LLLM98]{LLLM}
F. Lancien, G. Lancien and C. Le Merdy.
\newblock A joint functional calculus for sectorial operators with commuting resolvents.
\newblock Proc. London Math. Soc. (3) 77 (1998), no. 2, 387--414.

\bibitem[LeM99]{LeM99}
C. Le Merdy.
\newblock $H^\infty$-functional calculus and applications to maximal regularity.
\newblock Semi-groupes d'op\'erateurs et calcul fonctionnel (Besan\c con, 1998), 41--77, Publ. Math. UFR Sci. Tech. Besan\c con, 16, Univ. Franche-Comt\'e, Besan\c con, 1999.

\bibitem[LM03]{LM03}%
C. Le Merdy.
\newblock Two results about $H^\infty$ functional calculus on analytic UMD Banach spaces.
\newblock J. Aust. Math. Soc. 74 (2003), no. 3, 351--378.



\bibitem[LiT79]{LiT79}
J. Lindenstrauss and L. Tzafriri.
\newblock Classical Banach spaces. II. Function spaces.
\newblock Ergeb. Math. Grenzgeb., 97 [Results in Mathematics and Related Areas]
Springer-Verlag, Berlin-New York, 1979.

\bibitem[LPG20]{LPG20}
A. Lischke, G. Pang, M. Gulian et al.
\newblock What is the fractional Laplacian? A comparative review with new results.
\newblock J. Comput. Phys. 404 (2020), 109009, 62 pp.

\bibitem[Lor19]{Lor19}
E. Lorist.
\newblock The $\ell^s$-boundedness of a family of integral operators on UMD Banach function spaces.
\newblock Positivity and noncommutative analysis, 365--379, Trends Math., Birkh\"auser/Springer, Cham, 2019.























 






















\bibitem[Mar15]{Mar15}
A. Martini.
\newblock Spectral multipliers on Heisenberg-Reiter and related groups.
\newblock Ann. Mat. Pura Appl. (4) 194 (2015), no. 4, 1135--1155.

\bibitem[MaM13]{MaM13}%
A. Martini and D. M\"uller.
\newblock $L^p$ spectral multipliers on the free group $N_{3,2}$.
\newblock Studia Math. 217 (2013), no. 1, 41--55.

\bibitem[MaM14a]{MaM14a}
A. Martini and D. M\"uller.
\newblock A sharp multiplier theorem for Grushin operators in arbitrary dimensions.
\newblock Rev. Mat. Iberoam. 30 (2014), no. 4, 1265--1280.

\bibitem[MaM14b]{MaM14b}
A. Martini and D. M\"uller.
\newblock Spectral multiplier theorems of Euclidean type on new classes of two-step stratified groups.
\newblock Proc. Lond. Math. Soc. (3) 109 (2014), no. 5, 1229--1263.

\bibitem[MaM16]{MaM16}%
A. Martini and D. M\"uller.
\newblock Spectral multipliers on 2-step groups: topological versus homogeneous dimension.
\newblock Geom. Funct. Anal. 26 (2016), no. 2, 680--702.

\bibitem[MMN23]{MMN23}
A. Martini, D. M\"uller and S. Nicolussi Golo.
\newblock Spectral multipliers and wave equation for sub-Laplacians: lower regularity bounds of Euclidean type.
\newblock J. Eur. Math. Soc. (JEMS) 25 (2023), no.3, 785--843.



\bibitem[MaM90]{MaM90}%
G. Mauceri and S. Meda.
\newblock Vector-valued multipliers on stratified groups.
\newblock Rev. Mat. Iberoamericana 6 (1990), no. 3-4, 141--154.

\bibitem[Med90]{Med90}
S. Meda.
\newblock A general multiplier theorem.
\newblock Proc. Amer. Math. Soc. 110 (1990), no. 3, 639--647.








\bibitem[MuS94]{MuS94}%
D. M\"uller and E. M. Stein.
\newblock On spectral multipliers for Heisenberg and related groups.
\newblock J. Math. Pures Appl. (9) 73 (1994), no. 4, 413--440.

\bibitem[MPR07]{MPR07}%
D. M\"uller, M. M. Peloso and F. Ricci.
\newblock $L^p$-spectral multipliers for the Hodge Laplacian acting on 1-forms on the Heisenberg group.
\newblock Geom. Funct. Anal. 17 (2007), no. 3, 852--886.

\bibitem[MPR15]{MPR15}
D. M\"uller, M. M. Peloso, F. Ricci.
\newblock Analysis of the Hodge Laplacian on the Heisenberg group.
\newblock Mem. Amer. Math. Soc. 233 (2015), no. 1095.





 






\bibitem[Nie23]{Nie23}%
L. Niedorf.
\newblock Restriction type estimates and spectral multipliers on M\'etivier groups.
\newblock Preprint, arXiv:2304.12960.








 




\bibitem[OnQ89]{OnQ89}
C. W. Onneweer and T. S. Quek.
\newblock Multipliers on weighted $L^p$-spaces over locally compact Vilenkin groups.
\newblock Proc. Amer. Math. Soc. 105 (1989), no. 3, 622--631.






 



\bibitem[Onn88]{Onn88}
C. W. Onneweer.
\newblock Weak $L^p$-spaces and weighted norm inequalities for the Fourier transform on locally compact Vilenkin groups.
\newblock Proceedings of the analysis conference, Singapore 1986, 191--201, North-Holland Math. Stud., 150, North-Holland, Amsterdam, 1988.














\bibitem[Pis16]{Pis16}
G. Pisier.
\newblock Martingales in Banach spaces.
\newblock Cambridge Studies in Advanced Mathematics, 155. Cambridge University Press, Cambridge, 2016.


%










\bibitem[PSW18]{PSW}
J. Pr\"uss, G. Simonett and M. Wilke.
\newblock Critical spaces for quasilinear parabolic evolution equations and applications.
\newblock J. Differential Equations 264 (2018), no. 3, 2028--2074.

\bibitem[Que87]{Que87}
T. S. Quek.
\newblock Weighted norm inequalities for the Fourier transform on certain totally disconnected groups.
\newblock Proc. Amer. Math. Soc. 101 (1987), no. 1, 113--121.




\bibitem[RaV99]{RaV99}
D. Ramakrishnan and R. J. Valenza.
\newblock Fourier Analysis on Number Fields.
\newblock Graduate Texts in Mathematics 186, Springer-Verlag, New York Inc. (1999).












\bibitem[Rob00]{Rob00}
A. Robert.
\newblock A course in $p$-adic analysis.
\newblock Grad. Texts in Math., 198. Springer-Verlag, New York, 2000.

\bibitem[RoZ08]{RoZ08}
J. J. Rodriguez-Vega and W. A. Zuniga-Galindo.
\newblock Taibleson operators, $p$-adic parabolic equations and ultrametric diffusion.
\newblock Pacific J. Math. 237 (2008), no. 2, 327--347.




\bibitem[Rub86]{Rub86}
J. L. Rubio de Francia.
\newblock Martingale and integral transforms of Banach space valued functions.
\newblock Probability and Banach spaces (Zaragoza, 1985), 195--222. Lecture Notes in Math., 1221. Springer-Verlag, Berlin, 1986.



















 










\bibitem[SWS90]{SWS90}
F. Schipp, W. R. Wade and P. Simon.
\newblock Walsh series. An introduction to dyadic harmonic analysis. 
\newblock Adam Hilger, Ltd., Bristol, 1990. 



\bibitem[SeS89]{SeS89}
A. Seeger and C. D. Sogge.
\newblock On the boundedness of functions of (pseudo-) differential operators on compact manifolds.
\newblock Duke Math. J. 59 (1989), no.3, 709--736.



\bibitem[Spe70]{Spe70}
R. Spector.
\newblock Sur la structure locale des groupes ab\'{e}liens localement compacts. (French).
\newblock Bull. Soc. Math. France Suppl. Mém. 24 (1970), 94 pp. 

\bibitem[SBSW15a]{SBSW15a}
R. S. Stankovic, P. L. Butzer, F. Schipp and W. R. Wade.
\newblock Dyadic Walsh analysis from 1924 onwards--Walsh-Gibbs-Butzer dyadic differentiation in science. Vol. 1. Foundations. A monograph based on articles of the founding authors, reproduced in full. In collaboration with the co-authors: Weiyi Su, Yasushi Endow, Sandor Fridli, Boris I. Golubov, Franz Pichler and Kees Onneweer.
\newblock Atlantis Studies in Mathematics for Engineering and Science, 12. Atlantis Press, Paris, 2015.

\bibitem[SBSW15b]{SBSW15b}
R. S. Stankovic, P. L. Butzer, F. Schipp and W. R. Wade.
\newblock Dyadic Walsh analysis from 1924 onwards--Walsh-Gibbs-Butzer dyadic differentiation in science. Vol. 2. Extensions and generalizations. A monograph based on articles of the founding authors, reproduced in full. In collaboration with the co-authors: Weiyi Su, Yasushi Endow, Sandor Fridli, Boris I. Golubov, Franz Pichler and Kees Onneweer.
\newblock Atlantis Studies in Mathematics for Engineering and Science, 13. Atlantis Press, Paris, 2015.







\bibitem[Sik92]{Sik92}%
A. Sikora.
\newblock Multiplicateurs associ\'es aux souslaplaciens sur les groupes homog\`enes. (French).
\newblock C. R. Acad. Sci. Paris Sér. I Math. 315 (1992), no.4, 417--419.







\bibitem[Sog17]{Sog17}%
C. Sogge.
\newblock Fourier integrals in classical analysis. Second edition.
\newblock Cambridge Tracts in Math., 210. Cambridge University Press, Cambridge, 2017.




 




%
%





%






\bibitem[Ste70]{Ste70}
E. M. Stein.
\newblock Topics in harmonic analysis related to the Littlewood-Paley theory.
\newblock Annals of Mathematics Studies, No. 63 Princeton University Press,
Princeton, N.J.; University of Tokyo Press, Tokyo, 1970.















\bibitem[Sti19]{Sti19}
P. R. Stinga.
\newblock User's guide to the fractional Laplacian and the method of semigroups.
\newblock Handbook of fractional calculus with applications. Vol. 2, 235--265. De Gruyter, Berlin, 2019.


\bibitem[Tai68]{Tai68}
M. Taibleson.
\newblock Harmonic analysis on $n$-dimensional vector spaces over local fields. I. Basic results on fractional integration.
\newblock Math. Ann. 176 (1968), 191--207.

\bibitem[Tai70a]{Tai70a}
M. Taibleson.
\newblock Harmonic analysis on $n$-dimensional vector spaces over local fields. II. Generalized Gauss kernels and the Littlewood-Paley function.
\newblock Math. Ann. 186 (1970), 1--19.

\bibitem[Tai70b]{Tai70b}
M. Taibleson.
\newblock Harmonic analysis on $n$-dimensional vector spaces over local fields. III. Multipliers.
\newblock Math. Ann. 187 (1970), 259--271.

\bibitem[Tai75]{Tai75}
M. Taibleson.
\newblock Fourier analysis on local fields.
\newblock Princeton University Press, Princeton, N.J.; University of Tokyo Press, Tokyo, 1975.





























\bibitem[Vil63]{Vil63}
N. J. Vilenkin.
\newblock On a class of complete orthonormal systems.
\newblock Amer. Math. Soc. Transl. (2) 28 (1963) 1--35.

\bibitem[Vla99]{Vla99}
V.S. Vladimirov.
\newblock Tables of Integrals of Complex-valued Functions of $p$-Adic Arguments.
\newblock Preprint, arXiv:math-ph/9911027.

\bibitem[VlV89]{VlV89}
V. S. Vladimirov and I. V. Volovich.
\newblock $p$-adic quantum mechanics.
\newblock Comm. Math. Phys. 123 (1989), no. 4, 659--676.


\bibitem[VVZ94]{VVZ94}
V. S. Vladimirov, I. V. Volovich and E. I. Zelenov.
\newblock $p$-adic analysis and mathematical physics.
\newblock Ser. Soviet East European Math., 1, World Scientific Publishing Co., Inc., River Edge, NJ, 1994.




 



 







\bibitem[Wei07]{Wei07}
F. Weisz.
\newblock Almost everywhere convergence of Banach space-valued Vilenkin-Fourier series.
\newblock Acta Math. Hungar. 116 (2007), no. 1-2, 47--59.


\bibitem[Wil98]{Wil98}
J. S. Wilson.
\newblock Profinite groups.
\newblock London Mathematical Society Monographs. New Series, 19. The Clarendon Press, Oxford University Press, New York, 1998.






\bibitem[Xu07]{Xu07}%
X. Xu.
\newblock New proof of the H\"ormander multiplier theorem on compact manifolds without boundary.
\newblock Proc. Amer. Math. Soc. 135 (2007), no. 5, 1585--1595.








\bibitem[ZaZ23]{ZaZ23}
B. A. Zambrano-Luna and W. A. Zuniga-Galindo.
\newblock $p$-adic cellular neural networks: applications to image processing.
\newblock Phys. D 446 (2023), Paper No. 133668, 11 pp.

\bibitem[Zun16]{Zun16}
W. A. Zuniga-Galindo.  
\newblock Pseudodifferential equations over non-Archimedean spaces.
\newblock Lecture Notes in Mathematics, 2174. Springer, Cham, 2016.

\bibitem[Zun22a]{Zun22a}
W. A. Zuniga-Galindo.
\newblock Non-Archimedean quantum mechanics via quantum groups.
\newblock Nuclear Phys. B 985 (2022), Paper No. 116021, 21 pp.

\bibitem[Zun22b]{Zun22b}
W. A. Zuniga-Galindo.
\newblock $p$-adic analysis: a quick introduction.
\newblock $p$-adic analysis, arithmetic and singularities, 177--221. Contemp. Math., 778. American Mathematical Society, 2022.

\bibitem[Zun24a]{Zun24a}
W. A. Zuniga-Galindo.  
\newblock Deep Neural Networks: A Formulation Via non-Archimedean Analysis.
\newblock Preprint, arXiv:2402.00094.

\bibitem[Zun24b]{Zun24b}
W. A. Zuniga-Galindo.  
\newblock $\wp$-Adic Quantum Mechanics, the Dirac Equation, and the violation of Einstein causality.
\newblock J. Phys. A 57 (2024), no. 30, Paper No. 305301, 29 pp.

\bibitem[Zun25]{Zun25}
W. A. Zuniga-Galindo.  
\newblock $p$-Adic Analysis.
\newblock  Walter de Gruyter GmbH, Berlin/Boston, 2025.

\end{thebibliography}
\end{document}